\documentclass[a4paper, 10pt, reqno]{amsart}

\usepackage[utf8]{inputenc}
\usepackage[T1]{fontenc}
\usepackage{amsthm}
\usepackage{amsmath}
\usepackage{amssymb}
\usepackage[inline]{enumitem}
\usepackage[dvips]{graphicx}
\usepackage{comment}
\usepackage{hyperref}
\usepackage{url}
\usepackage{stackrel}
\usepackage{fancyhdr}
\usepackage{mathrsfs}
\usepackage{stmaryrd}
\usepackage{pdfcomment}
\usepackage[usenames,dvipsnames]{xcolor}
\usepackage[normalem]{ulem}
\usepackage[abbrev]{amsrefs}
\usepackage{tikz-cd}
\usepackage[all]{xy}

\newtheorem{theorem}{Theorem}[section]
\newtheorem{lemma}[theorem]{Lemma}
\newtheorem{proposition}[theorem]{Proposition}
\newtheorem{corollary}[theorem]{Corollary}

\newtheorem*{claim*}{\textsc{Claim}}

\theoremstyle{definition}
\newtheorem{definition}[theorem]{Definition}
\newtheorem{example}[theorem]{Example}

\newtheorem{remark}[theorem]{\textbf{Remark}}

\definecolor{blue-url}{RGB}{0,0,100}
\definecolor{yellow-highlight}{RGB}{255,249,193}
\definecolor{dark-green}{RGB}{48,104,68}

\hypersetup{
    pdftitle={On the Arithmetic of Power Monoids},
    pdfauthor={Austin Antoniou and Salvatore Tringali},
    pdfmenubar=false,
    pdffitwindow=true,
    pdfstartview=FitH,
    colorlinks=true,
    linkcolor=blue-url,
    citecolor=green,
    urlcolor=blue-url
}

\DeclareMathSymbol{\widehatsym}{\mathord}{largesymbols}{"62}

\renewcommand{\bf}{\mathbf}
\renewcommand{\emptyset}{\varnothing}

\renewcommand{\setminus}{\smallsetminus}
\newcommand{\m}{\mathsf{m}}

\providecommand{\AAc}{\mathscr{A}}

\providecommand{\CCc}{\mathscr{C}}

\providecommand{\LLc}{\mathscr{L}}

\providecommand{\NNb}{{\mathbf{N}}}
\providecommand{\ord}{{\rm ord}}

\providecommand{\PPc}{\mathcal{P}}

\providecommand{\ZZb}{\mathbf{Z}}
\providecommand{\ZZc}{\mathcal{Z}}

\providecommand\llb{\llbracket}
\providecommand\rrb{\rrbracket}
\providecommand\fin{{\rm fin}}
\providecommand\funt{{{\rm fin},\times}}
\providecommand\fun{{{\rm fin},1}}

\newcommand{\gen}[1]{\langle #1 \rangle}

\newcommand\circledot{%
  \mathop{\ooalign{\hss$\cdot$\hss\cr%
  \kern0.03ex\raise0.15ex\hbox{\scalebox{0.75}{$\circ$}}}}}

\newcommand\circleplus{%
  \mathop{\ooalign{\hss$+$\hss\cr%
  \kern0.468ex\raise0.165ex\hbox{\scalebox{0.75}{$\circ$}}}}}

\newcommand{\fixed}[2][1]{%
  \begingroup
  \spaceskip=#1\fontdimen2\font minus \fontdimen4\font
  \xspaceskip=0pt\relax
  #2%
  \endgroup
}

\begin{document}
\title{On the Arithmetic of Power Monoids \\ and Sumsets in Cyclic Groups}

\author{Austin A.~Antoniou}
\address{Department of Mathematics, The Ohio State University | 231 W.~18th Avenue, Columbus, OH 43202, USA}
\email{antoniou.6@osu.edu}
\urladdr{https://people.math.osu.edu/antoniou.6/}
\author{Salvatore Tringali}
\address{Institute for Mathematics and Scientific Computing, University of Graz, NAWI Graz | Heinrichstr.~36, 8010 Graz, Austria}
\curraddr{School of Mathematical Sciences, Hebei Normal University | No.~20 Road East, 2nd Southern Ring, Yuhua District | Shijiazhuang, Hebei, 050024 China}
\email{salvo.tringali@gmail.com}
\urladdr{https://imsc.uni-graz.at/tringali}


\thanks{
	A.A. was supported by a Rhodus Graduate Fellowship from the Department of Mathematics at the Ohio State University, and S.T. by the Austrian Science Fund (FWF), Project No.~M 1900-N39.}
\begin{abstract}
	\noindent{}Let $H$ be a multiplicatively written monoid with identity $1_H$ (in particular, a group), and denote by $\mathcal P_{\rm fin,\times}(H)$ the monoid obtained by endowing the collection of all finite subsets of $H$ containing a unit with the operation of setwise multiplication $(X,Y) \mapsto \{xy: x \in X, y \in Y\}$. We study fundamental features of the arithmetic of this and related structures, with a focus on the submonoid, $\mathcal P_{\text{fin},1}(H)$, of $\mathcal P_{\text{fin},\times}(H)$ consisting of all finite subsets of $H$ containing the identity.
	
	Among others, we establish that $\mathcal{P}_{\text{fin},1}(H)$ is atomic (i.e., each non-unit is a product of atoms) if and only if $1_H \ne x^2 \ne x$ for every $x \in H \setminus \{1_H\}$.
Then we prove that $\mathcal{P}_{\text{fin},1}(H)$ is BF (i.e., it is atomic and every element has factorizations of bounded length) if and only if $H$ is torsion-free; and we show how to transfer these conclusions from $\mathcal P_{\text{fin},1}(H)$ to $\mathcal P_{\text{fin},\times}(H)$ through the machinery of equimorphisms.
	
Next, we introduce a suitable notion of ``minimal factorization'' (and investigate its behavior with respect to equimorphisms) to account for the fact that monoids may have non-trivial idempotents, in which case standard definitions from Factorization Theory degenerate. Accordingly, we obtain necessary and sufficient conditions for $\mathcal P_{\text{fin},\times}(H)$ to be BmF (meaning that each non-unit has at least one minimal factorization and all such factorizations are bounded in length); and for $\mathcal{P}_{\text{fin},1}(H)$ to be BmF, HmF (i.e., a BmF-monoid where all the minimal factorizations of a given element have the same length), or minimally factorial (i.e., a BmF-monoid where each non-unit element has an essentially unique minimal factorization). Finally, we prove how to realize certain intervals as sets of minimal lengths in $\mathcal P_{\text{fin},1}(H)$.
	
Many proofs come down to considering sumset decompositions in cyclic groups, so giving rise to an intriguing interplay with Arithmetic Combinatorics.
\end{abstract}

\maketitle
\thispagestyle{empty}

\section{Introduction}\label{sec:intro}
By and large, Factorization Theory is about generalizations of the Fundamental Theorem of Arithmetic from the integers to other settings: Slightly more precisely, it can be understood as the study of certain properties of rings and, more generally, monoids where factorization of elements into ``irreducibles'' need not be unique, see the proceedings \cite{And97,ChGl00,Ch05,ChFoGeOb16}, the surveys \cite{BaCh11,BaWi13,Ge16c}, or the volumes \cite{FoHoLu13,GeHK06}.
While the focus has been so far on integral domains and commutative cancellative monoids, the field has recently witnessed an increasing interest for settings where cancellativity or commutativity may not be satisfied, see \cite{BaGe14, FTr17, GeZh19, Sm19, Tr18} and references therein.
In the present work, we further contribute to this line of research, by inquiring into the arithmetic of a new class of ``highly non-cancellative'' monoids recently introduced in \cite{FTr17} to serve as a bridge between Factorization Theory and Arithmetic Combinatorics, with emphasis on the ``structural theory'' of sumsets and product sets in groups \cite{Na96,TaVu06,GeRu09,Gry13}.

More in detail, let $H$ be a monoid (in particular, a group) with identity $1_H$ (see \S{ }\ref{sec:fundamentals} for basic notation and terminology). The set of all non-empty finite subsets of $H$ is then also a monoid, denoted by $\mathcal P_\fin(H)$ and called the \emph{power monoid} of $H$, when endowed with the operation of setwise multiplication
\[
(X,Y) \mapsto XY := \{xy: x \in X,\, y \in Y\}.
\]
Moreover, the set $\mathcal P_\funt(H)$ of all finite subsets of $H$ containing a unit is a submonoid  of $\mathcal P_\fin(H)$, and so is the set $\mathcal P_\fun(H)$ of all finite subsets of $H$ containing $1_H$: We will refer to the former as the \emph{restricted power monoid} of $H$, and to the latter as the \emph{reduced power monoid} of $H$.

After recalling some central ideas and objects from Factorization Theory in \S{ }\ref{sec:fundamentals}, we begin considering power monoids in \S{ }\ref{sec:atomicity}.
We explore the interplay between the restricted and the reduced power monoids, and we conclude that, under mild assumptions on $H$, their arithmetic is essentially the same.
Specifically, we give necessary and sufficient conditions under which $\mathcal P_{\fin,1}(H)$ is atomic (Theorem \ref{th:atomicity}); and by a combinatorial argument simply based on the Pigeonhole Principle, we conclude that $\PPc_\fun(H)$ is BF if and only if $H$ is torsion-free (parts \ref{it:thm:BF-torsion(i)} and \ref{it:thm:BF-torsion(ii)} of Theorem \ref{thm:BF-torsion}). Then we discuss how to transfer these conclusions from $\mathcal P_{\text{fin},1}(H)$ to $\mathcal P_{\text{fin},\times}(H)$ through the machinery of equimorphisms (Definition \ref{def:equimorphism}), under the hypothesis that $H$ is Dedekind-finite (Proposition \ref{prop:funt&fun-have-the-same-system-of-lengths}, Example \ref{exa:no-dedekind-finiteness}, and Theorem \ref{thm:BF-torsion}\ref{it:thm:BF-torsion(iii)}).

In \S{ }\ref{sec:minimal-factorizations}, we address a well-known limitation of Factorization Theory as developed in the classical setting: Finer arithmetic properties (e.g., boundedness of factorization lengths) are completely precluded by the existence of non-trivial idempotent elements. In the setting of power monoids this issue is nearly unavoidable (Example \ref{unbounded-fzn}).
We are thus motivated to introduce \emph{minimal factorizations} (Definitions \ref{def:preorder} and \ref{def:min-fac}), that is, a refinement of the classical notion of factorization which, on the one hand, 
circumvents the problem and, on the other hand, is no different from the usual notion in the commutative, cancellative setting (Proposition \ref{prop:min-basics}\ref{it:prop:min-basics(iv)}).

In particular, we first investigate how minimal factorizations behave with respect to equimorphisms (Propositions \ref{prop:min-equi} and \ref{prop:comm-pm} and Example \ref{exa:strict-inclusion}). Then we obtain necessary and sufficient conditions for $\mathcal P_{\fin,1}(H)$ to be BmF (meaning that the monoid is atomic and the minimal factorizations of a fixed element are all bounded in length), HmF (i.e., a BmF-monoid where all the minimal factorizations of a given element have the same length), or minimally factorial (i.e., a BmF-monoid where each non-unit element has an essentially unique minimal factorization); and for $\mathcal{P}_{\text{fin},\times}(H)$ to be BmF (Theorems \ref{BmF-char} and \ref{prop:HF-exp-3} and Corollary \ref{cor:when-reduced-pm-is-minimally-factorial}).
Finally in \S{ }\ref{sec:cyclic-case}, we focus on $\PPc_\fun(H)$ for the special case when $H$ is a finite cyclic group or is isomorphic to the additive monoid of non-negative integers.
This allows us to show (Theorem \ref{th:interval-lengths}) that, for general $H$, certain intervals can be always realized as ``sets of minimal lengths'' in $\mathcal P_{\fin,1}(H)$.

\section{Preliminaries}
\label{sec:fundamentals}
In this short section, we fix some definitions that we will need as we inquire into the algebraic and arithmetic structure of power monoids.
\subsection{Generalities}\label{subsec:generalities}
Unless noted otherwise, we reserve the letters $\ell$, $m$, and $n$ (with or without subscripts) for positive integers; and the letters $i$, $j$, and $k$ for non-negative integers.

We use $\mathbf N$ for the natural numbers (in particular, $0 \in \mathbf N$), and
for $a, b \in \mathbf R \cup \{\infty\}$ we define $\llb a, b \rrb := \mathbf Z \cap [a,b]$.
If $X$ is a set and $\mathcal{E}$ an equivalence on $X$, we denote by $\mathcal P(X)$ the power set of $X$ and by $\llb x \rrb_{\mathcal{E}}$ the class of an element $x \in X$ in the quotient $X/\mathcal{E}$.

When $n$ is understood from context, $\overline{k}$ will denote the residue class of $k$ modulo $n$.
Occasionally, we will also need to lift residues modulo $n$ back to the integers. So, if $a\in \ZZb/n\ZZb$ and there is no risk of confusion, we set $\hat{a} := \min (a \cap \mathbf N) \in \llb 0, n-1 \rrb$, and for $A\subseteq \ZZb/n\ZZb$ we define $\hat{A} := \{\hat{a}:a\in A \}$.

Given a set $\mathscr U$, we denote by $\mathscr{F}(\mathscr U)$ the free monoid with basis $\mathscr U$, and by $\varepsilon_\mathscr{U}$ the identity of $\mathscr{F}(\mathscr U)$. We assume $\mathscr U \subseteq \mathscr{F}(\mathscr U)$, and adopt the symbol $\ast$ for the operation of $\mathscr{F}(\mathscr U)$. We call an element of $\mathscr{F}(\mathscr U)$ a $\mathscr U$-word, and $\varepsilon_\mathscr{U}$ the empty $\mathscr U$-word. 
	For all $\mathfrak z \in \mathscr{F}(\mathscr U)$ and $n \in \mathbf N^+$, we take $\mathfrak z^{\ast 0} := \varepsilon_\mathscr{U}$, $\mathfrak z^{\ast 1} := \mathfrak z$, and $\mathfrak z^{\ast n} := \mathfrak z^{\ast(n-1)} \ast \mathfrak z$; and we define the \emph{word length}, $\|\mathfrak z\|_\mathscr{U}$, of $\mathfrak z$ (relative to the basis $\mathscr U$) as follows: If $\mathfrak z = \varepsilon_\mathscr{U}$, then $\|\mathfrak z\|_\mathscr{U} := 0$; otherwise, there are (uniquely) de\-ter\-mined $z_1, \ldots, z_n \in \mathscr U$ such that $\mathfrak z = z_1 \ast \cdots \ast z_n$, and we let $\|\mathfrak z\|_\mathscr{U} := n$.
	Lastly, we say a $\mathscr U$-word $\mathfrak y$ is a \emph{subword} of a $\mathscr U$-word $\mathfrak z$ if $\mathfrak y = \varepsilon_\mathscr{U}$, or $\mathfrak z = z_1 \ast \cdots \ast z_n$ and $\mathfrak y = z_{i_1} \ast \cdots \ast z_{i_m}$ for some $i_1, \ldots, i_m \in \llb 1, n \rrb$ with $i_j < i_{j+1}$ for each $j \in \llb 1, m-1 \rrb$.

Further terminology and notations, if not explained, are standard or should be clear from the context.

\subsection{Basic definitions for monoids}
\label{subsec:basics-on-monoids}
Here and later, monoids will be usually written multiplicatively and, unless differently specified, need not have any special property (e.g., commutativity).

Let $H$ be a monoid with identity $1_H$. We denote by $H^\times$ and $\mathscr{A}(H)$, respectively, the \emph{set of units} (or \emph{invertible elements}) and the \emph{set of atoms} of $H$, where $a \in H$ is an atom if $a \notin H^\times$ and there do not exist $x, y\in H \setminus H^\times$ such that $a = xy$.
We say that $H$ is \emph{reduced} if $H^\times = \{1_H\}$; \emph{cancellative} if $xz = yz$ or $zx = zy$, for some $x, y, z \in H$, implies $x = y$; \emph{Dedekind-finite} if there do not exist $x, y \in H$ with $xy = 1_H \ne yx$; and \emph{unit-cancellative}
provided that $xy \ne x \ne yx$ for all $x, y \in H$ with $y \notin H^\times$. 

Given $x, y \in H$, we
write $x \mid_H y$, read ``$x$ divides $y$ (in $H$)'', if $y \in HxH$; and $x \simeq_H y$, read ``$x$ is \emph{associate} to $y$ (in $H$)'', if $y \in H^\times x H^\times$.
We set $\langle x \rangle_H := \{x^n: n \in \mathbf N^+\}$ and let $\ord_H(x)$ be the \emph{order} of $x$ (relative to $H$), that is, the cardinality of $\langle x \rangle_H$ (when $x$ is a unit, this coincides with the common group-theoretic sense of ``order'').
Lastly, we call $x$ an \emph{idempotent} element (of $H$) if $x^2=x$, and we take a submonoid $M$ of $H$ to be \emph{divisor-closed} if $x \in M$ whenever $x \mid_H y$ and $y \in M$.

It is perhaps worth mentioning that there are, in principle, several options for a definition of ``$x$ divides $y$'' (and hence of ``divisor-closedness'') in a non-commutative setting, cf. \cite[\S{ }5]{BaSm15}.
The definition we are choosing here is consistent with \cite{FTr17} and fits very well with our approach to the study of factorization in a setting that is not only non-commutative, but also non-cancellative.
\subsection{Factorizations and lengths}\label{subsec:factorizations}
We let the \emph{factorization homomorphism} of $H$ be the unique (monoid) homomorphism $\pi_H: \mathscr{F}(H) \to H$ such that $\pi_H(x) = x$ for all $x \in H$, and we write $\mathscr{C}_H$ for the smallest monoid con\-gru\-ence on $\mathscr{F}(\mathscr{A}(H))$ for which the following holds:
\begin{enumerate}[label={$\bullet$}]
	\item If $\mathfrak a = a_1 \ast \cdots \ast a_m$ and $\mathfrak b = b_1 \ast \cdots \ast b_n$ are, re\-spec\-tive\-ly, non-empty $\mathscr{A}(H)$-$\fixed[0.2]{\text{ }}$words of length $m$ and $n$, then $(\mathfrak a, \mathfrak b) \in \mathscr{C}_H$ if and only if $\pi_H(\mathfrak a) = \pi_H(\mathfrak b)$, $m = n$, and $a_1 \simeq_H b_{\sigma(1)}, \ldots, a_n \simeq_H b_{\sigma(n)}$ for some permutation $\sigma$ of the discrete interval $\llb 1, n \rrb$.
\end{enumerate}
So, if $\mathfrak a = a_1 \ast \cdots \ast a_n$ is a non-empty $\mathscr{A}(H)$-word of length $n$ and $H$ is reduced and commutative, then
\[
\llb \mathfrak a \rrb_{\mathscr{C}_H} = \bigl\{a_{\sigma(1)} \ast \cdots \ast a_{\sigma(n)}: \sigma \text{ is a permutation of } \llb 1, n \rrb\bigr\}.
\]
In addition, we define, for every $x \in H$, the \emph{set of factorizations of $x$} by
\[
\mathcal{Z}_H(x) := \pi_H^{-1}(x) \cap \mathscr{F}(\mathscr{A}(H));
\]
consequently, we take
\[
\mathsf{Z}_H(x) := \mathcal Z_H(x)/\mathscr C_H = \bigl\{ \llb \mathfrak{a} \rrb_{\mathscr{C}_H} : \mathfrak{a}\in \mathcal{Z}_H(x) \bigr\}
\quad\text{and}\quad
{\sf L}_H(x) := \bigl\{\|\mathfrak a\|_H: \mathfrak a \in \mathcal{Z}_H(x)\bigr\} \subseteq \mathbf N
\]
to be the sets of \emph{factorization classes} and \emph{factorization lengths} of $x$, respectively.
Then we say that $H$ is
\begin{itemize}
\item \emph{atomic} if $\mathcal{Z}_H(x) \ne \emptyset$ for every $x \in H \setminus H^\times$; 
\item FF (respectively, BF) if $H$ is atomic and $\mathsf{Z}_H(x)$ (respectively, $\mathsf{L}_H(x)$) is finite for each $x \in H$; 
\item \textup{HF} (respectively, \emph{factorial}) if $|\mathsf{L}_H(x)| = 1$ (respectively, $|\mathsf{Z}_H(x)| = 1$) for all $x \in H \setminus H^\times$.
\end{itemize} 
Of course, we have that
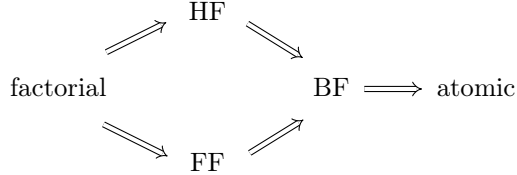
\[
\begin{tikzcd}[arrows=Rightarrow, row sep=0.5cm]
	&
	\!\text{HF}\arrow[shorten <= 2pt,shorten >= 2pt, shift right=-1ex]{dr}
	&
	\\
	\text{factorial}
	\arrow[shorten <= 4pt, shorten >= 4pt, shift right=-1ex]{ur}
	\arrow[shorten <= 7pt, shorten >= 4pt, shift left=-1ex]{dr}
	&  &
	\text{BF} \arrow[shorten <= 1pt, shorten >= 1pt]{r}
	&
	\text{atomic}  \\
	&
	\!\text{FF} \arrow[shorten <= 3pt,shorten >= 3pt, shift left=-1ex]{ur}
	&
\end{tikzcd}
\] 
Finally, we let the \emph{system of sets of lengths} of $H$ be the family 
\[
\mathscr{L}(H) := \{\mathsf{L}_H(x): x\in H \} \setminus \{\emptyset\} \subseteq \mathcal P(\mathbf N).
\]
\subsection{Literature}
\label{sub:other-factorizations}
Our approach to factorization in possibly non-cancellative or non-commutative monoids is borrowed from \cite{FTr17}, where one can read thoroughly about differences and similarities with 
the classical approach to factorization in commutative and cancellative monoids (and hence in integral domains) pursued by A.~Geroldinger and F.~Halter-Koch in \cite{GeHK06}, and with the more recent approach to factorization in cancellative but possibly non-commutative monoids set forth by N.\,R.~Baeth and D.~Smertnig in \cite{BaSm15}; in particular, see \cite[Remarks 2.6 and 2.7]{FTr17}.

This said, there are many previous entries in the literature that have treated aspects of factorization theory in commutative (unital) rings with non-trivial zero divisors. Most notably, D.\,D.~Anderson and collaborators have extensively studied factorizations in commutative rings corresponding to notions of ``associate'' and ``irreducible'' other than the ones adopted in the present paper, see e.g. \cite{AndMar85a,AndMar85b,AnVL96,AnVL97,AgAnVL01,ChunAnd11,ChAnVLe11}. Below we review these alternative definitions and contrast them with our approach. 

To start with, let $R$ be a commutative ring and denote by $R^\times$ the set of units of the multiplicative monoid of $R$. Given $x, y \in R$, we say in the parlance of \cite[Definition 2.1]{AnVL96} that 
\begin{itemize}
	\item $x$ is \emph{associate} to $y$ (in $R$), written $x \sim_R y$, if $xR = yR$;
	\item $x$ is \emph{strongly associate} to $y$, written $x \approx_R y$, if $x \in yR^\times$ (by Proposition \ref{prop:unit-adjust}\ref{it:prop:unit-adjust(0)}, this is equivalent to $x$ being associate (as per \S{ }\ref{subsec:basics-on-monoids}) to $y$ in the multiplicative monoid of $R$);
	\item $x$ is \emph{very strongly associate} to $y$, written $x \cong_R y$, if $x \sim_R y$ and one of the following holds:
	\begin{enumerate}[label = {\rm (\roman{*})}] 
		\item $x = y = 0_R$ (where $0_R$ is the zero of $R$);
		\item $x \neq 0_R$ and if $x = yz$ for some $z \in R$ then $z \in R^\times$.
	\end{enumerate}
\end{itemize}
Accordingly, one has three notions of ``irreducible'', see \cite[Definition 2.4]{AnVL96}. To wit, an element $a \in R$ is 
\begin{itemize}
	\item \emph{irreducible} if $a \notin R^\times$ and $a=xy$ for some $x, y \in R$ implies that $a \sim_R x$ or $a \sim_R y$;
	\item \emph{strongly irreducible} if $a \notin R^\times$ and $a=xy$ for some $x, y \in R$ implies that $a \approx_R x$ or $a \approx_R y$;
	\item \emph{very strongly irreducible} if $a \notin R^\times$ and $a=xy$ for some $x, y \in R$ implies that $a \cong_R x$ or $a \cong_R y$.
\end{itemize}
It is obvious that very strongly irreducible elements of $R$ are strongly irreducible, and strongly irreducible elements are irreducible. In general, none of these implications can be reversed, see the paragraph after the proof of Theorem 2.12 in \cite{AnVL96}. However, we get by \cite[Theorem 2.2(3)]{AnVL96} that the three notions coincide when $R$ is \emph{pr\'esimplifiable} in the sense of \cite{Bou74a}, meaning that if $xy = x$ for some $x, y \in R$ then $x = 0_R$ or $y \in R^\times$ (e.g., this is the case when $R$ is an integral domain). Moreover, \cite[Theorem 2.5]{AnVL96} yields that a \emph{non-zero} element of $R$ is strongly irreducible if and only if it is an atom of the multiplicative monoid of $R$.

Putting it all together, we thus see that the ring $R$ is \emph{very strongly atomic} in the sense of \cite[Definition 3.1]{AnVL96} if and only if one of the following holds:
\begin{enumerate}[label={(\small{A}\arabic{*})}]
\item\label{it:condition(A1)} $R$ has non-trivial zero divisors and the multiplicative monoid of $R$ is atomic (as per \S{ }\ref{subsec:factorizations});
\item\label{it:condition(A2)} $R$ is an integral domain and $R \setminus \{0_R\}$ is an atomic monoid under multiplication.
\end{enumerate}
Similarly, $R$ is a \emph{bounded factorization ring} in the sense of \cite[Definition 3.8]{AnVL96}, a \emph{half-factorial ring} in the sense of \cite[p. 87]{AxFoRoSt03}, a \emph{finite factorization ring} in the sense of \cite[Definition 6.5]{AnVL96}, or a unique factorization ring in the sense of \cite[Definition 4.3]{AnVL96} if and only if one of conditions \ref{it:condition(A1)} and \ref{it:condition(A2)} in the above is satisfied with ``atomic'' replaced, respectively, by ``BF'', ``HF'', ``FF'', or ``factorial''.

As long as the scope is restricted to commutative rings, it is therefore possible to compare our approach to factorization with others based on irreducibles, strong irreducibles, or even alternative ``elementary factors'' (including the ones considered by C.\,R.~Fletcher \cite{Fl69}, A.~Bouvier \cite{Bou74a, Bou74b}, and S.~Galovich \cite{Ga78}) by referring to \cite{AnVL96,ChunAnd11}, where these comparisons are worked out in great detail. (Incidentally, it appears that ``irreducibles'' in the sense of \cite{Ga78} are non-units, although this is not explicitly stated by Galovich in his paper.) See also \cite[Theorem 3.4 and Corollary 3.5]{BaBuMi17} for a couple of results of a more arithmetic flavor concerning lengths of factorizations into irreducibles in commutative rings of the form $D/xD$, where $D$ is a principal ideal domain and $x$ is a non-zero, non-unit element of $D$. 
\section{Atomicity and bounded factorization in power monoids}
\label{sec:atomicity}
Here we embark on the study of the (arithmetic and algebraic) structure of power monoids.
We begin with some elementary but helpful observations we will often use without comment.
\begin{proposition}\label{prop:unit-adjust}
Let $H$ be a monoid.
The following hold:
\begin{enumerate}[label={\rm (\roman{*})}]
	\item\label{it:prop:unit-adjust(0)} If $u, v \in H^\times$ then $uv \in H^\times$, and the converse holds whenever $H$ is Dedekind-finite.
	\item\label{it:prop:unit-adjust(0b)} If $\mathscr A(H)$ is non-empty or $H$ is commutative or unit-cancellative, then $H$ is Dedekind-finite.
\item\label{it:prop:unit-adjust(i)} If $a\in \mathscr{A}(H)$ and $u, v \in H^\times$, then $uav\in \mathscr{A}(H)$.
\item\label{it:prop:unit-adjust(ii)} If $x\in H\setminus H^\times$ and $u, v \in H^\times$, then $\mathsf{L}_H(uxv) = \mathsf{L}_H(x)$.
\end{enumerate}
\end{proposition}
\begin{proof}
See \cite[parts (i), (ii), and (iv) of Lemma 2.2, and Proposition 2.30]{FTr17}.
\end{proof}

\begin{proposition}\label{prop:pm-arith}
Let $H$ be a monoid. The following hold:
	\begin{enumerate}[label={\rm (\roman{*})}]
	\item\label{it:prop:pm-arith(i)} If $X_1,\ldots,X_n\in \PPc_{\fin,1}(H)$, then $X_1 \cup \cdots \cup X_n \subseteq X_1 \cdots X_n$.
	\item\label{it:prop:pm-arith(ii)} If $u,v \in H^\times$ and $X_1,\ldots,X_n\in \PPc_{\fin,\times}(H)$, then $|uX_1 \cdots X_nv| = |X_1 \cdots X_n| \ge \max_{1 \le i \le n} |X_i|$.
	\item\label{it:prop:pm-arith(iii)} If $K$ is a submonoid of $H$, then $\PPc_\fun(K)$ is a divisor-closed submonoid of $\PPc_\fun(H)$. \textup{(}Note that the conclusion is valid regardless of whether $K$ itself is divisor-closed.\textup{)}
	\item \label{it:prop:pm-arith(iv)} $\PPc_\fun(H)$ is a reduced monoid and $\PPc_\fin(H)^\times=\PPc_\funt(H)^\times = \bigl\{ \{u\}: u\in H^\times\bigr\}$.
	\item \label{it:prop:pm-arith(v)} $\mathscr{A}(\mathcal P_{\funt}(H)) \subseteq H^\times \mathscr{A}(\mathcal P_{\fun}(H)) H^\times$.
	\end{enumerate}
\end{proposition}
\begin{proof}
\ref{it:prop:pm-arith(i)} is trivial, upon considering that $(X \cdot 1_H) \cup (1_H \cdot Y) \subseteq XY$ for all $X, Y \in \mathcal P_\fun(H)$; \ref{it:prop:pm-arith(ii)} is a direct consequence of \ref{it:prop:pm-arith(i)} and the fact that the function $X \to H: x \mapsto uxv$ is injective for all $u, v \in H^\times$ and $X \subseteq H$; and \ref{it:prop:pm-arith(iii)} and \ref{it:prop:pm-arith(iv)} are immediate from \ref{it:prop:pm-arith(i)} and \ref{it:prop:pm-arith(ii)}.

As for \ref{it:prop:pm-arith(v)}, let $A\in \mathscr{A}(\PPc_\funt(H))$.
	Because $A$ contains a unit of $H$, there is $u\in H^\times$ such that $1_H \in uA$.
	Then $uA$ is an element of $\PPc_\fun(H)$, and by Proposition \ref{prop:unit-adjust}\ref{it:prop:unit-adjust(i)} it is also an atom of $\PPc_\funt(H)$.
	Thus, if $X,Y\in \PPc_\fun(H) \subseteq \PPc_\funt(H)$ and $uA = XY$, then $X$ or $Y$ is the identity of $\PPc_\fun(H)$.
	This means that $uA$ is an atom of $\PPc_\fun(H)$, and hence $A=u^{-1}(uA)\in H^\times \mathscr{A}(\PPc_\fun(H))$, as wished.
\end{proof}
Our ultimate goal is, for an arbitrary monoid $H$, to investigate factorizations in $\PPc_\fin(H)$. However, this is a difficult task in general, due to a variety of ``pathological situations'' that might be hard to classify in a satisfactory way, see e.g. \cite[Remark 3.3(ii)]{FTr17}.

In practice, it is more convenient to start with $\PPc_\fun(H)$ and then lift arithmetic results from $\PPc_\fun(H)$ to $\PPc_\funt(H)$, a point of view which is corroborated by the simple consideration that $\PPc_\fin(H) = \PPc_\funt(H)$ whenever $H$ is a group (i.e., in the case of greatest interest in Arithmetic Combinatorics).

In turn, we will see that studying the arithmetic of $\PPc_\funt(H)$ is tantamount to studying that of $\PPc_\fun(H)$, in a sense to be made precise presently.
To do so in as all-encompassing a way as possible, we recall from \cite[Definition 3.2]{Tr18} a notion which formally packages the idea that, under suitable conditions, arithmetic may be transferred from one monoid to another.
\begin{definition}\label{def:equimorphism}
Let $H$ and $K$ be monoids, and let $\varphi$ be a function from $H$ to $K$. We denote by $\varphi^*$ the unique monoid homomorphism  $\mathscr{F}(H) \to \mathscr{F}(K)$ such that $\varphi^\ast(x) = \varphi(x)$ for every $x \in H$, and we call $\varphi$ an \emph{equimorphism} (from $H$ to $K$) if $\varphi$ is a monoid homomorphism $H \to K$ and the following hold:
\begin{enumerate}[label={({\small{E}}\arabic{*})}]
\item\label{def:equimorphism(E1)} $\varphi^{-1}(K^\times)\subseteq H^\times$;
\item\label{def:equimorphism(E2)} $\varphi$ is \emph{atom-preserving}, meaning that $\varphi(\mathscr{A}(H)) \subseteq \mathscr{A}(K)$;
\item\label{def:equimorphism(E3)} If $x\in H$ and $\mathfrak{b}\in \mathcal{Z}_K(\varphi(x))$ is a non-empty $\mathscr A(K)$-word, then $\varphi^*(\mathfrak{a}) \in \llb \mathfrak{b} \rrb_{\mathscr{C}_K}$ for some $\mathfrak{a}\in \mathcal{Z}_H(x)$.
\end{enumerate}
Moreover, we say that $\varphi$ is \emph{essentially surjective} if $K = K^\times \varphi(H)K^\times$.
\end{definition}
\begin{proposition}\label{prop:equimorphism}
Let $H$ and $K$ be monoids and $\varphi:H\to K$ an equimorphism. The following hold:
\begin{enumerate}[label={\rm (\roman{*})}]
\item\label{it:prop:equimorphism(i)} $\mathsf{L}_H(x) = \mathsf{L}_K(\varphi(x))$ for all $x\in H\setminus H^\times$.
\item\label{it:prop:equimorphism(ii)} If $\varphi$ is essentially surjective, then for all $y\in K\setminus K^\times$ there is $x\in H\setminus H^\times$ with $\mathsf{L}_K(y) = \mathsf{L}_H(x)$.
\end{enumerate}
\end{proposition}
\begin{proof}
	See \cite[Theorem 2.22(i)]{FTr17} and \cite[Theorem 3.3(i)]{Tr18}.
\end{proof}
\begin{proposition}\label{prop:funt&fun-have-the-same-system-of-lengths}
	Let $H$ be a Dedekind-finite monoid. The following hold:
	\begin{enumerate}[label={\rm (\roman{*})}]
	\item\label{it:prop:funt&fun-have-the-same-system-of-lengths(i)} The natural embedding $\jmath: \PPc_\fun(H)\hookrightarrow \PPc_\funt(H)$ is an essentially surjective equimorphism.
	\item\label{it:prop:funt&fun-have-the-same-system-of-lengths(ii)} $\mathscr{A}(\mathcal P_{\funt}(H)) = H^\times \mathscr{A}(\mathcal P_{\fun}(H)) H^\times$.
	\item\label{it:prop:funt&fun-have-the-same-system-of-lengths(iii)} $\mathsf{L}_{\mathcal P_{\fun}(H)}(X) = \mathsf{L}_{\mathcal P_{\funt}(H)}(X)$ for every $X \in \mathcal P_{\fun}(H)$.
	\item\label{it:prop:funt&fun-have-the-same-system-of-lengths(iv)} $\mathscr L(\mathcal P_{\funt}(H)) = \mathscr L(\mathcal P_{\fun}(H))$.
	\end{enumerate}
\end{proposition}
\begin{proof}
	In view of Proposition \ref{prop:equimorphism}, parts \ref{it:prop:funt&fun-have-the-same-system-of-lengths(iii)} and \ref{it:prop:funt&fun-have-the-same-system-of-lengths(iv)} are immediate from \ref{it:prop:funt&fun-have-the-same-system-of-lengths(i)}.
	Moreover, the inclusion from left to right in \ref{it:prop:funt&fun-have-the-same-system-of-lengths(ii)} is precisely the content of Proposition \ref{prop:pm-arith}\ref{it:prop:pm-arith(v)}, and the other inclusion will follow from \ref{it:prop:funt&fun-have-the-same-system-of-lengths(i)} and Propositions \ref{prop:unit-adjust}\ref{it:prop:unit-adjust(i)} and \ref{prop:pm-arith}\ref{it:prop:pm-arith(iv)}.
	Therefore, we focus on \ref{it:prop:funt&fun-have-the-same-system-of-lengths(i)} for the remainder of the proof.
	
	\ref{it:prop:funt&fun-have-the-same-system-of-lengths(i)} By Proposition \ref{prop:pm-arith}\ref{it:prop:pm-arith(iv)}, $\jmath$ satisfies \ref{def:equimorphism(E1)}.
	Moreover, $\jmath$ is essentially surjective, as any $X\in \mathcal P_\funt(H)$ contains a unit $u \in H^\times$, so $u^{-1}X\in \mathcal P_\fun(H)$ and $X = u(u^{-1}X)$ is associate to an element of $\mathcal P_\fun(H)$. 

	To prove \ref{def:equimorphism(E2)}, let $A\in\mathscr{A}(\mathcal P_\fun(H))$.
	We aim to show that $A$ is an atom of $\mathcal P_\funt(H)$. Suppose that $A = XY$ for some $X,Y\in \mathcal P_\funt(H)$. Then there are $x \in X$ and $y \in Y$ with
	$xy=1_H$; and using that $H$ is Dedekind-finite, we get from Proposition \ref{prop:unit-adjust}\ref{it:prop:unit-adjust(0)} that $x,y\in H^\times$.
	It follows that 
	\[
	A = XY = (Xx^{-1})(xY)
	\quad\text{and}\quad
	Xx^{-1}, xY\in \mathcal P_\fun(H). 
	\]
	But then
	$Xx^{-1} = \{1_H\}$ or $xY = \{1_H\}$, since $\mathcal P_\fun(H)$ is a reduced monoid and $A$ is an atom of $\mathcal P_\fun(H)$.
	So, $X$ or $Y$ is a $1$-element subset of $H^\times$, and hence $A \in \mathscr A(\mathcal P_\funt(H))$.

	It remains to show that $\jmath$ satisfies \ref{def:equimorphism(E3)}.  Pick $X \in \mathcal P_\fun(H)$. If $X = \{1_H\}$, the conclusion holds vacuously. Otherwise, let $\mathfrak{b} := B_1*\cdots*B_n \in \mathcal Z_{\mathcal P_\fun(H)}(X)$. Then there are $u_1\in B_1,\ldots, u_n\in B_n$ such that $1_H = u_1\cdots u_n$; and as in the proof of \ref{def:equimorphism(E2)}, it must be that $u_1,\ldots, u_n\in H^\times$.
	Accordingly, we take, for every $i \in \llb 1, n \rrb$, $A_i := u_0 \cdots u_{i-1} B_i u_i^{-1} \cdots u_1^{-1}$, where $u_0 := 1_H$.
	Then 
	\[
	A_1 \cdots A_n = X
	\quad\text{and}\quad
	1_H \in A_1 \cap \cdots \cap A_n;
	\]
	and by Propositions \ref{prop:unit-adjust}\ref{it:prop:unit-adjust(i)} and \ref{prop:pm-arith}\ref{it:prop:pm-arith(v)}, $A_1, \ldots, A_n$ are atoms of $\mathcal P_{\fun}(H)$. This shows that $\mathfrak{a}:= A_1*\cdots * A_n \in \mathcal{Z}_{\mathcal P_\fun(H)}(X)$. Since $A_i \simeq_{\mathcal P_\funt(H)} B_i$ for each $i \in \llb 1, n \rrb$ (by construction), we thus conclude that $\mathfrak{a}$ is $\mathscr{C}_{\mathcal P_{\funt}(H)}$-congruent to $\mathfrak{b}$, as wished.
\end{proof} 
The next example proves that Dedekind-finiteness is, to some extent, necessary for Proposition \ref{prop:funt&fun-have-the-same-system-of-lengths}\ref{it:prop:funt&fun-have-the-same-system-of-lengths(ii)}, and hence for the subsequent conclusions. 
\begin{example}\label{exa:no-dedekind-finiteness}
Let $\mathcal{B}$ be the set of all binary sequences $\mathfrak{s}: \mathbf N^+ \to \{0,1\}$, and let $H$ denote the monoid of all functions $\mathcal{B} \to \mathcal{B}$ under composition.  We will write $H$ multiplicatively; so, if $f, g \in H$ then $fg$ is the map $\mathcal{B} \to \mathcal{B}: \mathfrak{s} \mapsto f(g(\mathfrak s))$.
Further, let $n\ge 5$ and consider the functions 
\begin{align*}
L: \mathcal{B} \to\mathcal{B}: (a_1,a_2,\dots) &\mapsto (a_2,a_3,\dots)  &\,\textrm{(left shift);}\\
R: \mathcal{B} \to\mathcal{B}: (a_1,a_2,\dots) &\mapsto (0,a_1,a_2,\dots) &\,\textrm{(right shift);}\\
P: \mathcal{B} \to\mathcal{B}: (a_1,a_2,\dots) &\mapsto (a_{n},a_1,\dots, a_{n-1}, {a_{n+1}, a_{n+2}},\ldots) &\,\textrm{(cycle the first $n$ terms).}
\end{align*}
In particular, $P \in H^\times$. Also, $LR = \operatorname{id}_\mathcal{B}$ but $RL \neq \operatorname{id}_\mathcal{B}$; whence $H$ is not Dedekind-finite, and neither $R$ nor $L$ is invertible. With this in mind, we will prove that $A:= \{L,P\} \cdot \{R,P\} = \{\operatorname{id}_\mathcal{B}, LP, PR, P^2 \}$ is an atom of $\PPc_\fun(H)$, although it is not, by construction, an atom of $\PPc_\funt(H)$.

Indeed,
assume $A = XY$ for some $X,Y\in \PPc_\fun(H)$.
Then $X,Y\subseteq A$, and it is clear that $P^2 \ne PRLP$, or else $RL = {\rm id}_\mathcal{B}$ (a contradiction). Similarly, $PRPR \ne P^2 \ne LPLP$; otherwise, $P = RPR$ and hence $R$ is invertible, or $P = LPL$ and $L$ is invertible (again a contradiction). Lastly, we see that $P^2 \ne LP^2 R$, by applying both $P^2$ and $LP^2 R$ to the constant sequence $(1, 1, \ldots)$.

It follows that $P^2$ must belong to $X$ or $Y$, but not to both (which is the reason for choosing $n\ge 5$).
Accordingly, let $P^2 \in X \setminus Y$ (the other case is analogous).
Then $Y = \{\operatorname{id}_\mathcal{B}\}$, since one can easily check that $P^2LP, P^3R \notin A$, by noting that the action of $P^2LP$ and $P^3R$ differ from that of $A$ on the sequences $(1,1,\ldots)$ and $(1,0,1,1,\ldots)$. This makes $A$ an atom of $\PPc_\fun(H)$.
\end{example}
We get from Proposition \ref{prop:funt&fun-have-the-same-system-of-lengths} that studying factorization properties of $\PPc_\fun(H)$ is sufficient for studying corresponding properties of $\PPc_\funt(H)$, at least in the case when $H$ is Dedekind-finite.
Thus, as a starting point in the investigation of the arithmetic of $\PPc_\fun(H)$, one might wish to give a comprehensive description of
the atoms of $\PPc_\fun(H)$.
This is however an overwhelming task even in specific cases (e.g., when $H$ is the additive group of the integers), let alone the general case. Nevertheless, we can obtain basic information about $\mathscr A(\PPc_\fun(H))$ in full generality.
\begin{lemma}\label{lem:2-elt-atoms}
Let $H$ be a monoid and $x \in H \setminus \{1_H\}$.
The following hold:
\begin{enumerate}[label={\rm (\roman{*})}]
\item\label{it:lem:2-elt-atoms(i)} The set $\{1_H, x\}$ is an atom of $\PPc_\fun(H)$ if and only if $1_H \ne x^2 \ne x$.
\item\label{it:lem:2-elt-atoms(ii)} If $x^2=1_H$ or $x^2=x$, then $\{1_H,x\}$ factors into a product of atoms neither in $\PPc_{\fin,1}(H)$ nor in $\PPc_\funt(H)$.
\end{enumerate}
\end{lemma}
\begin{proof}
\ref{it:lem:2-elt-atoms(i)} If $x^2 = 1_H$ or $x^2 = x$, then it is clear that $\{1_H,x\} = \{1_H,x\}^2$, and therefore $\{1_H,x\}$ is not an atom of $\mathcal P_\fun(H)$. As for the converse, assume that $\{1_H,x\} = YZ$ for some non-units $Y,Z \in \mathcal P_\fun(H)$. Then we get from Proposition \ref{prop:pm-arith} that $Y$ and $Z$ are $2$-element sets, namely, $Y = \{1_H, y\}$ and $Z = \{1_H, z\}$ with $y,z \in H \setminus \{1_H\}$. Hence $\{1_H,x\} = YZ = \{1_H,y,z,yz\}$, and immediately this implies $x=y=z$. Therefore, $\{1_H,x\} = \{1_H,x,x^2\}$, which is only possible if $x^2 = 1_H$ or $x^2 = x$.

\ref{it:lem:2-elt-atoms(ii)} Suppose that $x^2 = 1_H$ or $x^2 = x$. Then the calculation above shows that $\{1_H,x\} = \{1_H,x\}^2$ and there is no other decomposition of $\{1_H,x\}$ into a product of non-unit elements of $\mathcal P_{\fin,1}(H)$. So, $\{1_H,x\}$ is a non-trivial idempotent (hence, a non-unit) and has no factorization into atoms of $\mathcal P_{\fin,1}(H)$.

It remains to prove the analogous statement for $\mathcal P_{\fin,\times}(H)$.
{Assume to the} contrary that $\{1_H,x\}$ factors into a product of $n$ atoms of $\mathcal P_{\fin,\times}(H)$ for some $n \in \mathbf N^+$. Then $n \ge 2$, since $\{1_H, x\}$ is a non-trivial idempotent (and hence not an atom itself). Consequently, we can write $\{1_H,x\} = YZ$, where $Y$ is an atom and $Z$ a non-unit of $\mathcal P_{\fin,\times}(H)$. In particular, we get from parts \ref{it:prop:pm-arith(i)}, \ref{it:prop:pm-arith(ii)}, and \ref{it:prop:pm-arith(iv)} of Proposition \ref{prop:pm-arith} that both $Y$ and $Z$ are $2$-element sets, say, $Y = \{u, y\}$ and $Z = \{v, z\}$. It is then immediate that there are only two possibilities: $1_H$ is the product of two units from $Y$ and $Z$, or the product of two non-units from $Y$ and $Z$.
Without loss of generality, we are thus reduced to considering the following cases.
\vskip 0.1cm
\textsc{Case 1:} $uv=1_H$. Then $uz \ne 1_H$ (or else $z = u^{-1} = v$, contradicting the fact that $Z$ is a $2$-element set). So $uz=x$, and similarly $yv=x$. Then $y = xu = uzu$ and $z = xv = vyv$, and therefore
\[
\{u,y\} = \{u,uzu\} = \{1_H,uz\}\, \{u\} = \{1_H,x\}\,\{u\} = \{u,y\}\, \{vu,zu\}.
\]
However, this shows that $\{u,y\}$ is not an atom of $\PPc_\funt(H)$, in contrast with our assumptions.
\vskip 0.1cm
\textsc{Case 2:}
$yz = 1_H$ and $y,z\in H\setminus H^\times$. Then $u,v\in H^\times$, by the fact that $\{u, y\}, \{v, z\} \in \mathcal P_{\fin,\times}(H)$; and we must have $uz=x$, for $uz=1_H$ would yield $z = u^{-1}\in H^\times$.
In particular, $x = uz$ is not a unit in $H$, so $uv=1_H$ and we are back to the previous case.
\end{proof}
We have just seen that, to even {\it hope} for $\PPc_\fun(H)$ to be atomic, we need that the ``bottom layer'' of $2$-element subsets of $H$ consists only of atoms; perhaps surprisingly, it will turn out that such a condition is also sufficient.
Before proving this, we point out some structural implications of the fact that every non-identity element of $H$ is neither an idempotent nor a square root of $1_H$.
\begin{lemma}\label{lem:no-non-id-elts-of-small-order-implies-structure}
	Let $H$ be a monoid such that $1_H \ne x^2 \ne x$ for all $x\in H\setminus\{1_H\}$. The following hold:
	\begin{enumerate}[label={\rm (\roman{*})}]
		\item\label{it:lem:no-non-id-elts-of-small-order-implies-structure(i)}
		$H$ is Dedekind-finite.
		\item\label{it:lem:no-non-id-elts-of-small-order-implies-structure(ii)}
		If $x \in H \setminus \{1_H\}$ and $\gen{x}_H$ is finite, then $x \in H^\times$ and $\gen{x}_H$ is a cyclic group of order $\ge 3$.
	\end{enumerate}
\end{lemma}
\begin{proof}
	\ref{it:lem:no-non-id-elts-of-small-order-implies-structure(i)}
	Let $y,z\in H$ such that $yz = 1_H$. Then $(zy)^2 = z(yz)y = zy$, and since $H$ has no non-trivial i\-dem\-po\-tents, we conclude that $zy=1_H$. Consequently, $H$ is Dedekind-finite.
	
	\ref{it:lem:no-non-id-elts-of-small-order-implies-structure(ii)}
	This is an obvious consequence of \cite[Ch. V, Exercise 4, p. 68]{Whitelaw}, according to which every finite semigroup has an idempotent.
	The proof is short, so we give it here for the sake of self-containedness.
	
	Because $\gen{x}_H$ is finite, there exist $n, k \in \mathbf N^+$ such that $x^n = x^{n+k}$, and by induction this implies that $x^n = x^{n+hk}$ for all $h \in \mathbf N$. Therefore, we find that
	\[
	(x^{nk})^2 = x^{2nk} = x^{(k+1)n}x^{(k-1)n} = x^n x^{(k-1)n} = x^{nk}.
	\]
	But $H$ has no non-trivial idempotents, thus it must be the case that $x^{nk}=1_H$. That is, $x$ is a unit of $H$, and we have $x^{-1} = x^{nk-1} \in \gen{x}_H$. So, $\gen{x}_H$ is a (finite) cyclic group of order $\ge 3$.
\end{proof}

\begin{theorem}\label{th:atomicity}
	Let $H$ be a monoid. Then $\mathcal P_\fun(H)$ is atomic if and only if $1_H \ne x^2 \ne x$ for every $x \in H \setminus \{1_H\}$.
\end{theorem}
\begin{proof}
	The ``only if'' part is a consequence of Lemma \ref{lem:2-elt-atoms}\ref{it:lem:no-non-id-elts-of-small-order-implies-structure(ii)}.
	As for the other direction, assume that $1_H \ne x^2 \ne x$ for each $x\in H\setminus\{1_H\}$, and fix $X \in \mathcal P_{\fun}(H)$ with $|X| \ge 2$.
	We wish to show that 
	\[
	X = A_1 \cdots A_n, \quad\text{for some }A_1, \ldots, A_n \in \mathscr{A}(\mathcal P_{\fin,1}(H)).
	\]
	If $X$ is a $2$-element set, the claim is true by Lemma \ref{lem:2-elt-atoms}\ref{it:lem:no-non-id-elts-of-small-order-implies-structure(i)}. So let $|X| \ge 3$, and suppose inductively that every $Y \in \PPc_\fun(H)$ with $2 \le |Y| < |X|$ is a product of atoms. If $X$ is an atom, we are done.
	Otherwise, $X = A B$ for some non-units $A, B \in \PPc_\fun(H)$, and by symmetry we can assume $|X| \ge |A| \ge |B| \ge 2$. 
	
	If $|A| < |X|$, then both $A$ and $B$ factor into a product of atoms (by the inductive hypothesis), and so too does $X=A B$. Consequently, we are only left to consider the case when $|X| = |A|$. 
	
	For, we notice that 
	$
	A \cup B \subseteq A B = X 
	$
	(because $1_H \in A \cap B$), and this is only possible if $A = X$ (since $|A| = |X|$ and $A \subseteq X$). So, to summarize, we have that
	\begin{equation}\label{equ:containments}
		|X| \ge 3, \quad |B| \ge 2, \quad\text{and}\quad B \subseteq AB = X = A.
	\end{equation}
	In particular, since $B$ is not a unit of $\mathcal P_{\fin,1}(H)$, we can choose an element $b\in B\setminus\{1_H\} \subseteq A$. Hence, taking $A_b := A \setminus \{b\}$, we have $|A_b| < |A|$, and it is easy to check that $A_b B = A = X$ (in fact, $1_H$ is in $A_b \cap B$, and therefore we derive from \eqref{equ:containments} that $
	A_b B \subseteq A = A_b \cup \{b\} \subseteq A_b B \cup \{b\} \subseteq A_b B \cup B = A_b B$).
	
	If $|B|<|A|$, then we are done, because $A_b$ and $B$ are both products of atoms (by the inductive hypothesis), and thus so is $X = AB = A_bB$.
	Otherwise, it follows from \eqref{equ:containments} and the above that 
	\begin{equation}\label{equ:further-identities}
		X = A = B = A_b B \quad\text{and}\quad |A| \ge 3, 
	\end{equation}
	so we can choose an element $a \in A \setminus\{1_H,b\}$. Accordingly, set $B_a := B \setminus \{a\}$. Then $|B_a| < |B|$ (because $A = B$ and $a \in A$), and both $A_b$ and $B_a$ decompose into a product of atoms (again by induction). But this finishes the proof, since it is straightforward from \eqref{equ:further-identities} that $X = A = A_b B_a$ (indeed, $1_H \in A_b \cap B_a$ and $b \in B_a$, so we find that $
	A_b B_b \subseteq A = A_b \cup \{b\} \subseteq A_b B_a \cup \{b\} \subseteq A_b B_a \cup B_a = A_b B_a$).
\end{proof}

Now with Proposition \ref{prop:equimorphism} and Theorem \ref{th:atomicity} in hand, we can engage in a finer study of the arithmetic of power monoids; in particular, we may wish to study their (systems of) sets of lengths. However, we are immediately met with a ``problem'' (i.e., some sets of lengths are infinite in a rather trivial way):
\begin{example}\label{unbounded-fzn}
	Let $H$ be a monoid with an element $x$ of finite odd order $m \ge 3$, and set $X := \{x^k : k \in \mathbf N \}$. Then it is clear that $X$ is the setwise product of $n$ copies of $\{1_H, x\}$ for every $n \ge m$. This shows that the set of lengths of $X$ relative to $\PPc_\fun(H)$ contains $\llb m, \infty \rrb$ (and hence is infinite), since we know from Lemma \ref{lem:2-elt-atoms} that $\{1_H, x\}$ is an atom of $\PPc_\fun(H)$.
\end{example}
The nature of this problem is better clarified by our next result, and we will more thoroughly address it in \S{ }\ref{sec:minimal-factorizations}.
\begin{theorem}\label{thm:BF-torsion}
Let $H$ be a monoid. The following hold:
\begin{enumerate}[label={\rm (\roman{*})}]
\item\label{it:thm:BF-torsion(i)} If $H$ is torsion-free and $X\in \PPc_\fun(H)$, then $\sup \mathsf{L}_{\mathcal P_\fun(H)}(X) \le |X|^2-|X|$.
\item\label{it:thm:BF-torsion(ii)} $\PPc_\fun(H)$ is \textup{BF} if and only if $H$ is torsion-free.
\item\label{it:thm:BF-torsion(iii)} $\PPc_\funt(H)$ is \textup{BF} if and only if so is $\mathcal P_\fun(H)$.
\end{enumerate}
\end{theorem}

\begin{proof}
	\ref{it:thm:BF-torsion(i)}
	Set $n := |X| \in \mathbf N^+$, fix an integer $\ell \ge (n-1)n + 1$, and suppose for a contradiction that $X = A_1\cdots A_\ell$ for some $A_1,\dots, A_\ell\in \mathscr{A}(\PPc_\fun(H))$.
	By the Pigeonhole Principle, there are an element $x\in X$ and a subset $I \subseteq \llb 1,\ell \rrb$ such that $m := |I| \ge n$ and $x \in A_i$ for each $i\in I$.
	So, writing $I = \{i_1,\ldots, i_m\}$, we find that $x^k \in A_{i_1} \cdots A_{i_k} \subseteq A_1\cdots A_\ell = X$ for every $k \in \llb 1, m \rrb$, i.e., $\{1_H,x,\dots, x^m\} \subseteq X$.
	However, since $H$ is torsion-free, each power of $x$ is distinct, and hence $n = |X| \ge m+1 > n$ (a contradiction).
	
	\ref{it:thm:BF-torsion(ii)}
	First suppose for a contradiction that $\PPc_\fun(H)$ is BF and has an element $x$ of finite order $m$. Then $\PPc_\fun(H)$ is also atomic, and we know by Theorem \ref{th:atomicity} and Lemma \ref{lem:no-non-id-elts-of-small-order-implies-structure}\ref{it:lem:no-non-id-elts-of-small-order-implies-structure(ii)} that $x^m = 1_H$.
	If $m$ is even, then $(x^{m/2})^2 = 1_H$, contradicting the atomicity of $\PPc_\fun(H)$ since, by Theorem \ref{th:atomicity}, no non-identity element of $H$ can have order 2. 
	If $m$ is odd, then Example \ref{unbounded-fzn} shows that the set of lengths of $\{x^k : k \in \mathbf N \}$ is infinite, contradicting the assumption that $\PPc_\fun(H)$ is BF.
	
	Conversely, assume $H$ is torsion-free. Then all powers of non-identity elements are distinct, so Theorem \ref{th:atomicity} implies that $\PPc_\fun(H)$ is atomic, and \ref{it:thm:BF-torsion(i)} gives an explicit upper bound on the lengths of factorizations.
	
	\ref{it:thm:BF-torsion(iii)} 
	The ``only if'' part follows from \cite[Theorem 2.28(iv) and Corollary 2.29]{FTr17}, so suppose that $\PPc_\fun(H)$ is BF.
	Then $\PPc_\fun(H)$ is atomic, and hence, by Theorem \ref{th:atomicity}, $1_H \neq x^2 \neq x$ for all $x\in H\setminus\{1_H\}$.
	By Lemma \ref{lem:no-non-id-elts-of-small-order-implies-structure}\ref{it:lem:no-non-id-elts-of-small-order-implies-structure(i)}, this implies that $H$ is Dedekind-finite, so the natural embedding $\PPc_\fun(H)\hookrightarrow\PPc_\funt(H)$ is an essentially surjective equimorphism by Proposition \ref{prop:funt&fun-have-the-same-system-of-lengths}\ref{it:prop:funt&fun-have-the-same-system-of-lengths(i)}.
	The result is then an immediate consequence of Proposition  \ref{prop:funt&fun-have-the-same-system-of-lengths}\ref{it:prop:funt&fun-have-the-same-system-of-lengths(iv)}.
\end{proof}

\section{Minimal factorizations and conditions for bounded minimal lengths}
\label{sec:minimal-factorizations}
Example \ref{unbounded-fzn} indicates that, in the presence of torsion in the ground monoid $H$, sets of lengths in $\PPc_\fun(H)$ blow up in a predictable fashion, with the result that most of the invariants classically studied in Factorization Theory lose their significance. 
In the case of Example \ref{unbounded-fzn}, this phenomenon is due to the existence of non-trivial idempotents and has been previously addressed by many authors in the literature on \emph{commutative} rings and monoids (see Remarks \ref{rem:chun&anderson} and \ref{rem:geroldinger&lettl}).
Here we strive for a ``natural approach'' that applies to \emph{arbitrary} monoids, spurring us to consider a refinement of the notions introduced in \S{ }\ref{subsec:factorizations} and to investigate some of their fundamental properties (see, in particular, Definition \ref{def:min-fac} and Proposition \ref{prop:min-equi}), before focusing on the special case of power monoids.

\subsection{Minimal factorizations.}
\label{subsec:min-factorizations}
We start with the definition of a binary relation (in fact, a preorder) on the $\mathscr A(H)$-words of a monoid $H$ that we shall use to ``filter out the redundant factors'' that may contribute to the factorizations of an element of $H$
(recall that, given a set $X$, we denote by $\mathscr{F}(X)$ the free monoid with basis $X$ and by $\varepsilon_X$ the identity of $\mathscr{F}(X)$).
\begin{definition}\label{def:preorder}
	Let $H$ be a monoid. We denote by $\preceq_H$ 
	the binary relation on $\mathscr{F}(\mathscr A(H))$ determined by taking $\mathfrak a \preceq_H \mathfrak b$, for some  
	$\mathscr A(H)$-words $\mathfrak a$ and $\mathfrak b$, if and only if either $\mathfrak a = \mathfrak b = \varepsilon_{\mathscr{A}(H)}$; or $\mathfrak a$ and $\mathfrak b$ are non-empty words of length $m$ and $n$ respectively, say $\mathfrak a = a_1 \ast \cdots \ast a_m$ and $\mathfrak b = b_1 \ast \cdots \ast b_n$, such that $\pi_H(\mathfrak a) = \pi_H(\mathfrak b)$ and
	$a_{\sigma(1)} \simeq_H b_1$, \ldots, $a_{\sigma(m)} \simeq_H b_m$ for some injection $\sigma: \llb 1, m \rrb \to \llb 1, n \rrb$.
	
	We say an $\mathscr A(H)$-word $\mathfrak{a}$ is \emph{$\preceq_H$-minimal}, or simply \emph{minimal} (when no confusion can arise), if there is no $\AAc(H)$-word $\mathfrak{b}$ with $\mathfrak{b} \prec_H \mathfrak{a}$, where $\mathfrak{b} \prec_H \mathfrak{a}$ means that $\mathfrak{b} \preceq_H \mathfrak{a}$ but $\mathfrak{a} \not\preceq_H \mathfrak{b}$.
\end{definition}
\begin{remark}
The preorder $\preceq_H$ in Definition \ref{def:preorder} is closely related to the ``divides-up-to-permutation'' relation $\mid_p$ considered by N.\,R.~Baeth and D.~Smertnig in \cite[Definition 5.2(2)]{BaSm15}, when the latter is spe\-cial\-iz\-ed to the free monoid $\mathscr F(\mathscr A(H))$: More precisely, $\mathfrak a \preceq \mathfrak b$ if and only if $\mathfrak a \mid_p \mathfrak b$ and $\pi_H(\mathfrak a) = \pi_H(\mathfrak b)$.
\end{remark}
The next result highlights a few basic properties of the relation introduced in Definition \ref{def:preorder}.
\begin{proposition}\label{preorder-facts}
Let $H$ be a monoid, and let $\mathfrak{a},\mathfrak{b}\in \mathscr{F}(\AAc(H))$. The following hold:
\begin{enumerate}[label = {\textup{(\roman{*})}}]
\item\label{it:prop:preorder-facts(0)} $\preceq_H$ is a preorder \textup{(}i.e., a reflexive and transitive binary relation\textup{)} on $\mathscr{F}(\mathscr A(H))$.
\item\label{it:prop:preorder-facts(ii)} If $\mathfrak{a} \preceq_H \mathfrak{b}$ then $\| \mathfrak{a} \|_H \le \| \mathfrak{b} \|_H$.
\item\label{it:prop:preorder-facts(iv)} $\mathfrak{a} \preceq_H \mathfrak{b}$ and $\mathfrak{b} \preceq_H \mathfrak{a}$ if and only if $\mathfrak a \preceq_H \mathfrak b$ and $\|\mathfrak a\|_H = \|\mathfrak b\|_H$, if and only if $(\mathfrak{a},\mathfrak{b}) \in \CCc_H$.
\end{enumerate}
\end{proposition}
\begin{proof}
Points \ref{it:prop:preorder-facts(0)} and \ref{it:prop:preorder-facts(ii)} are straightforward from our definitions.

As for \ref{it:prop:preorder-facts(iv)}, set $h := \|\mathfrak a\|_H$ and $k := \|\mathfrak b\|_H$. 
By part \ref{it:prop:preorder-facts(ii)}, $\mathfrak{a} \preceq_H \mathfrak{b}$ and $\mathfrak{b} \preceq_H \mathfrak{a}$ only if  $\mathfrak a \preceq_H \mathfrak b$ and $h = k$; and it is immediate to check that $(\mathfrak{a},\mathfrak{b}) \in \CCc_H$ implies $\mathfrak{a} \preceq_H \mathfrak{b}$ and $\mathfrak{b} \preceq_H \mathfrak{a}$. 
So, to finish the proof, assume that $\mathfrak a \preceq_H \mathfrak b$ and $h = k$. We only need to show that $(\mathfrak a, \mathfrak b) \in \mathscr C_H$. 
For, we have (by definition) that $\mathfrak a \preceq_H \mathfrak b$ if and only if $\pi_H(\mathfrak a) = \pi_H(\mathfrak b)$ and there is an injection $\sigma: \llb 1, h \rrb \to \llb 1, k \rrb$ 
such that $a_i \simeq_H b_{\sigma(i)}$ for every $i \in \llb 1, h \rrb$. But $\sigma$ is actually a bijection (because $h = k$), and we can thus conclude that $(\mathfrak a, \mathfrak b) \in \mathscr C_H$.
\end{proof}
\begin{definition}\label{def:min-fac}
Let $H$ be a monoid, and let $x \in H$. We refer to a $\preceq_H$-minimal $\mathscr A(H)$-word $\mathfrak{a}$ such that $x = \pi_H(\mathfrak a)$ as a \emph{$\preceq_H$-minimal factorization} of $x$, or simply as a \emph{minimal factorization} of $x$ provided $H$ is clear from context. Accordingly, we denote by
\[
\mathcal{Z}_{H}^\m(x) := \left\{ \mathfrak{a}\in \mathcal{Z}_H(x): \mathfrak{a} \textrm{ is $\preceq_H$-minimal} \right\}
\quad\text{and}\quad
\mathsf{Z}_{H}^\m(x):= \mathcal{Z}_{H}^\m(x)/\mathscr{C}_H
\]
the set of $\preceq_H$-minimal factorizations and the set of $\preceq_H$-minimal factorization classes of $x$, respectively (cf. the definitions from \S{ }\ref{subsec:factorizations}). In addition, we take 
\[
\mathsf{L}_{H}^\m(x) := \left\{ \|\mathfrak{a}\|_H : \mathfrak{a} \in \mathcal{Z}_{H}^\m(x) \right\} \subseteq \mathbf N
\]
to be the set of \emph{$\preceq_H$-minimal factorization lengths} of $x$, and
\[
	\LLc^\m(H) := \left\{ \mathsf{L}_{H}^\m(x) : x\in H \right\} \subseteq \mathcal P(\mathbf N)
\]
to be the \emph{system of sets of $\preceq_H$-minimal lengths} of $H$. Lastly, we say that the monoid $H$ is
\begin{itemize}
	\item \textup{BmF} or \emph{bounded-minimally-factorial} (respectively, \textup{FmF} or \emph{finite-minimally-factorial}) if $\mathsf{L}_{H}^\m(x)$ (respectively, $\mathsf{Z}_{H}^\m(x)$) is finite and non-empty for every $x\in H\setminus H^\times$;
	
	\item \textup{HmF} or \emph{half-minimally-factorial} (respectively, \emph{minimally factorial}) if $\mathsf{L}_{H}^\m(x)$ (respectively, $\mathsf{Z}_{H}^\m(x)$) is a singleton for all $x \in H \setminus H^\times$.
\end{itemize}
Note that we may write $\mathcal{Z}^\m(x)$ for $\mathcal{Z}_{H}^\m(x)$, $\mathsf {L}^\m(x)$ for $\mathsf {L}_{H}^\m(x)$, etc. if there is no likelihood of confusion.
\end{definition}

\begin{remark} \label{rem:chun&anderson}
To the best of our knowledge, analogues of the notions introduced in Definition \ref{def:min-fac} have only been considered so far in a \emph{commutative} setting, with one significant example being offered by the work of S.~Chun, D.\,D.~Anderson, and S.~Valdes-Leon \cite{ChAnVLe11} on ``reduced factorizations''.

In detail, let $H$ be the multiplicative monoid of a (unital) ring $R$ and fix a set $\mathcal A \subseteq R$.
We say that a non-empty $\mathcal A$-word $a_1 \ast \cdots \ast a_n$ of length $n$ is a \emph{minimal $\mathcal A$-factorization} of an element $x \in R$ if $\pi_H(\mathfrak a) = x$ but $x \ne \pi_H(\mathfrak b)$ for every non-empty $\mathcal A$-word $\mathfrak b = b_1 \ast \cdots \ast b_m$ of length $m \le n-1$ for which there exists an injection $\sigma: \llb 1, m \rrb \to \llb 1, n \rrb$ such that $b_i \simeq_H a_{\sigma(i)}$ for each $i \in \llb 1, m \rrb$. 

A minimal $\mathscr A(H)$-factorization of a non-unit $x \in R$ is the same as a $\preceq_H$-minimal factorization of $x$ (as per Definition \ref{def:min-fac}). Moreover, it follows from \S{ }\ref{sub:other-factorizations} and Proposition \ref{prop:unit-adjust}\ref{it:prop:unit-adjust(ii)} that, if $R$ is a \emph{commutative} ring and $x$ is not the zero of $R$, then a minimal $\mathscr A(H)$-factorization of $x$ is, in the parlance of \cite[Definition 2.1 and \S{ }3]{ChAnVLe11}, essentially the same as a \emph{strongly $\mu$-reduced $\mu$-factorization} of $x$ into very strongly irreducible elements of $R$. 
Insofar as the discussion is restricted to commutative rings, one can thus refer to \cite{ChAnVLe11} and \cite{AxFoRoSt03} for a comparison of our approach to the study of ``minimal factorizations'' with others in the literature, including the one by C.\,R.~Fletcher \cite{Fl69} and generalizations thereof where the set $\mathcal A$ in the above consists of various types of ``irreducible elements'' of $R$ (cf. \S{ }\ref{sub:other-factorizations}).
\end{remark}

\begin{remark}\label{rem:geroldinger&lettl}
Another approach for managing the ``excess factorizations'' arising from the presence of torsion (though still in a commutative setting), was outlined by A.~Geroldinger and G.~Lettl in \cite{GeLe90}. 

In short, let $H$ be a \emph{commutative} monoid and denote by $\mathcal A$ the set of all $a \in H \setminus H^\times$ such that $b \mid_H a$ only if $b \in H^\times$ or $aH = bH$. Given $u \in H$, we define 
\[
\textrm{ind}_H^{\rm GL}(u) := \inf \{r\in \mathbf{N}: u^i H = u^j H \textrm{ for all } i, j \ge r\}.
\]
Accordingly, we take a \emph{\textup{GL}-factorization} of a non-unit $x \in H$ to be a non-empty $\mathcal A$-word $\mathfrak a = a_1 \ast \cdots \ast a_n$ such that $\pi_H(\mathfrak a) = x$ and $\mathsf v_a^H(\mathfrak a) \le \text{ind}_H^{\rm GL}(a)$ for every $a \in \mathcal A$, where
\[
\mathsf v_a^H(\mathfrak a) := \bigl|\{i \in \llb 1, n \rrb: a_i = a\}\bigr|.
\]
A \textup{GL}-factorization is fundamentally the same as the ``canonical form'' of a factorization in the sense of \cite{GeLe90}; and since it is easily checked that $\mathscr A(H) \subseteq \mathcal A$, every $\preceq_H$-minimal factorization is also a \textup{GL}-factorization. Moreover, the two notions coincide on the level of commutative, unit-cancellative monoids, in which case $\mathcal A = \mathscr A(H)$ and $\text{ind}_H^{\rm GL}(u) = \infty$ for every non-unit $u \in H$.
However, big differences exist in general. E.g., it follows from Lemma \ref{lem:2-elt-atoms}\ref{it:lem:2-elt-atoms(i)} and the above that $\{\bar{0}, \bar{1}\} \ast \{\bar{0}, \bar{2}\} \ast \{\bar{0}, \bar{3}\} \ast \{\bar{0}, \bar{4}\}$ is a \textup{GL}-factorization of $\mathbf Z/5\mathbf Z$ in the reduced power monoid of the cyclic group $(\mathbf Z/5\mathbf Z, +)$; but is not a minimal factorization as per Definition \ref{def:min-fac}, because $\mathbf Z/5\mathbf Z = \{\bar{0}, \bar{1}\} + \{\bar{0}, \bar{2}\} + \{\bar{0}, \bar{3}\}$.
\end{remark}
It is helpful, at this juncture, to observe some fundamental features of minimal factorizations.
\begin{proposition}\label{prop:min-basics}
	Let $H$ be a monoid and let $x\in H$. The following hold:
	\begin{enumerate}[label = {\rm (\roman{*})}]
		\item\label{it:prop:min-basics(i)} Any $\AAc(H)$-word of length $0$, $1$, or $2$ is minimal.
		\item\label{it:prop:min-basics(ii)} $\ZZc_H(x) \ne \emptyset$ if and only if  $\ZZc_H^\m(x) \ne \emptyset$.
		\item\label{it:prop:min-basics(iib)} If $\mathfrak a \in \mathcal Z_H^\m(x)$ and $(\mathfrak a, \mathfrak b) \in \mathscr C_H$, then $\mathfrak b \in \mathcal Z_H^\m(x)$.
		\item\label{it:prop:min-basics(iii)} If $K$ is a divisor-closed submonoid of $H$ and $x \in K$, then $\ZZc_K^\m(x) = \ZZc_H^\m(x)$ and $\mathsf L_K^\m(x) = \mathsf L_H^\m(x)$.
		\item\label{it:prop:min-basics(iv)} If $H$ is commutative and unit-cancellative, then $\ZZc_H^\m(x) = \ZZc_H(x)$, and hence $\mathsf L_H^\m(x) = \mathsf L_H(x)$.
	\end{enumerate}
\end{proposition}
\begin{proof}
	\ref{it:prop:min-basics(i)}, \ref{it:prop:min-basics(ii)}, and \ref{it:prop:min-basics(iib)} are an immediate consequence of parts \ref{it:prop:preorder-facts(ii)}-\ref{it:prop:preorder-facts(iv)} of Proposition \ref{preorder-facts} (in particular, note that, if $\mathfrak a$  is an $\mathscr A(H)$-word of length $1$, then $\pi_H(\mathfrak a)$ is an atom of $H$, and therefore $\pi_H(\mathfrak a) \ne \pi_H(\mathfrak b)$ for every  $\mathscr A(H)$-words $\mathfrak b$ of length $\ge 2$); and \ref{it:prop:min-basics(iii)} follows at once from considering that, if $K$ is a divisor-closed submonoid of $H$ and $x \in K$, then $\mathcal Z_K(x) = \mathcal Z_H(x)$ and $\mathsf L_K(x) = \mathsf L_H(x)$, see \cite[Proposition 2.21(ii)]{FTr17}.
	
	\ref{it:prop:min-basics(iv)} Assume $H$ is commutative and unit-cancellative. It suffices to check that no non-empty $\mathscr A(H)$-word $\mathfrak a$ has a proper subword $\mathfrak b$ with $\pi_H(\mathfrak a) \simeq_H \pi_H(\mathfrak b)$. Suppose to the contrary that there exist  $a_1,\dots,a_n \in \mathscr A(H)$ with $\prod_{i \in I} a_i \simeq_H a_1\cdots a_n$ for some $I \subsetneq \llb 1, n \rrb$. Since $H$ is commutative, we can assume without loss of generality that $I = \llb 1, k \rrb$ for some $k \in \llb 0, n-1 \rrb$. Then unit-cancellativity implies $a_{k+1}\cdots a_n \in H^\times$, and we get from parts \ref{it:prop:unit-adjust(0)} and \ref{it:prop:unit-adjust(0b)} of Proposition \ref{prop:unit-adjust} that $a_{k+1},\dots,a_n \in H^\times$, which is however impossible (by definition of an atom).	
\end{proof}
To further elucidate the behavior of minimal factorizations, we give an analogue of Proposition \ref{prop:unit-adjust}\ref{it:prop:unit-adjust(ii)} showing that multiplying a non-unit by units does not change its set of minimal factorizations.
\begin{lemma}\label{lem:min-unit-adjust}
	Let $H$ be a monoid, and fix $x\in H\setminus H^\times$ and $u,v\in H^\times$.
	Then there is a length-preserving bijection $\mathcal{Z}_H^\m(x)\to\mathcal{Z}_H^\m(uxv)$, and in particular $\mathsf L_H^\m(x) = \mathsf L_H^\m(uxv)$.
\end{lemma}
\begin{proof}
	Given $w, z \in H$ and a non-empty word $\mathfrak{z} = y_1 \ast \cdots \ast y_n \in \mathscr{F}(H)$ of length $n$, denote by $w\mathfrak{z}z$ the length-$n$ word $\bar{y}_1 \ast \cdots \ast \bar{y}_n \in \mathscr{F}(H)$ defined by taking $\bar{y}_1 := w y_1 z$ if $n = 1$, and $\bar{y}_1:= wy_1$, $\bar{y}_n := y_nz$, and $\bar{y}_i := y_i$ for all $i\in \llb 2,n-1\rrb$ otherwise. 
	We claim that the function 
	$$
	f: \mathcal{Z}_H^\m(x) \to \mathcal{Z}_H^\m(uxv): \mathfrak a \mapsto u\mathfrak a v
	$$
	is a well-defined length-preserving bijection. 
	In fact, it is sufficient to show that $f$ is well-defined, since this will in turn imply that the map $g: \mathcal{Z}_H^\m(uxv)\to\mathcal{Z}_H^\m(x):\mathfrak{b}\mapsto u^{-1}\mathfrak{b}v^{-1}$ is also well-defined (observe that $uxv \in H \setminus H^\times$ and $x = u^{-1} uxv v^{-1}$), and then it is easy to check that $g$ is the inverse of $f$.
	
	For the claim, let $\mathfrak a \in \mathcal{Z}_H^\m(x)$, and note that, by parts \ref{it:prop:unit-adjust(0)} and \ref{it:prop:unit-adjust(0b)} of Proposition \ref{prop:unit-adjust}, $\|\mathfrak a\|_H$ is a positive integer, so that  $\mathfrak a = a_1 \ast \cdots \ast a_n$ for some $a_1, \ldots, a_n \in \mathscr A(H)$. 
	In view of Proposition \ref{prop:unit-adjust}\ref{it:prop:unit-adjust(i)}, $u\mathfrak{a}v$ is a factorization of $uxv$, and we only need to verify that it is also $\preceq_H$-minimal. For,
suppose to the contrary that $\mathfrak b \prec_H u\mathfrak av$ for some $\mathfrak b \in \mathscr{F}(\mathscr A(H))$.
	Then $\pi_H(\mathfrak b) = \pi_H(u\mathfrak a v) = uxv$ and, by Proposition \ref{preorder-facts}\ref{it:prop:preorder-facts(iv)},
	$
	k := \|\mathfrak b\|_H \in \llb 1, n-1 \rrb
	$
	(recall that $uxv \notin H^\times$). So, $\mathfrak b = b_1 \ast \cdots \ast b_k$ for some atoms  $b_1, \ldots, b_k \in H$, and there exists an injection $\sigma: \llb 1, k \rrb \to \llb 1, n \rrb$ such that $b_i \simeq_H a_{\sigma(i)}$ for each $i \in \llb 1, k \rrb$.
	Define
	$
	\mathfrak{c} := u^{-1} \mathfrak b v^{-1}$. 
	
	By construction and Proposition \ref{prop:unit-adjust}\ref{it:prop:unit-adjust(i)}, there are  $c_1, \ldots, c_k \in \mathscr A(H)$ such that $\mathfrak c = c_1 \ast \cdots \ast c_k$; and it follows from the above that $\pi_H(\mathfrak c) = u^{-1} \pi_H(\mathfrak b) v^{-1} = x$ and $c_i \simeq_H a_{\sigma(i)}$ for every $i \in \llb 1, k \rrb$. Since $k < n$, we can thus conclude from Proposition \ref{preorder-facts}\ref{it:prop:preorder-facts(iv)} that $\mathfrak c \prec_H \mathfrak a$, contradicting the $\preceq_H$-minimality of $\mathfrak{a}$.
\end{proof}
We saw in the previous section that equimorphisms transfer factorizations between monoids (Proposition \ref{prop:equimorphism}).  
Equimorphisms have a similar compatibility with minimal factorizations, in the sense that an equimorphism also satisfies a ``minimal version'' of condition \ref{def:equimorphism(E3)} from Definition \ref{def:equimorphism}.
\begin{proposition}\label{prop:min-equi}
	Let $H$ and $K$ be monoids and $\varphi: H\to K$ an equimorphism. The following hold: 
	\begin{enumerate}[label={\rm (\roman{*})}]
		\item\label{it:prop:min-equi(i)} If $x\in H \setminus H^\times$ and $\mathfrak{b}\in \mathcal{Z}_K^\m(\varphi(x))$, then there is $\mathfrak{a}\in \mathcal{Z}_H^\m(x)$ with $\varphi^*(\mathfrak{a})\in \llb \mathfrak{b} \rrb_{\mathscr{C}_K}$.
		\item\label{it:prop:min-equi(ii)} $\mathsf{L}_K^\m(\varphi(x)) \subseteq \mathsf{L}_H^\m(x)$ for every $x\in H\setminus H^\times$.
		\item\label{it:prop:min-equi(iii)} If $\varphi$ is essentially surjective then, for all $y\in K\setminus K^\times$, there is $x\in H\setminus H^\times$ with $\mathsf{L}_K^\m(y) \subseteq \mathsf{L}_H^\m(x)$. 
	\end{enumerate}
\end{proposition}
\begin{proof}
	\ref{it:prop:min-equi(i)}
	Pick $x\in H \setminus H^\times$, and let $\mathfrak b \in \mathcal{Z}_K^\m(\varphi(x))$. Then $\mathfrak b \ne \varepsilon_{\mathscr A(K)}$, otherwise $\varphi(x) = \pi_K(\mathfrak b) = 1_K$ and, by \ref{def:equimorphism(E1)}, $x \in \varphi^{-1}(\varphi(x)) = \varphi^{-1}(1_K) \subseteq H^\times$ (a contradiction). Consequently, \ref{def:equimorphism(E3)} yields the existence of a factorization $\mathfrak{a}\in \mathcal{Z}_H(x)$ with $\varphi^*(\mathfrak{a}) \in \llb \mathfrak{b} \rrb_{\mathscr{C}_K}$, and it only remains to show that $\mathfrak a$ is $\preceq_H$-minimal. 
	
	{Note that} $n := \|\mathfrak a\|_H = \|\varphi^\ast(\mathfrak a)\|_K = \|\mathfrak b\|_K \ge 1$, and write $\mathfrak a = a_1 \ast \cdots \ast a_n$ and $\mathfrak b = b_1 \ast \cdots \ast b_n$, with $a_1, \ldots, a_n \in \mathscr A(H)$ and $b_1, \ldots, b_n \in \mathscr A(K)$. Then suppose to the contrary that $\mathfrak a$ is not $\preceq_H$-minimal, i.e., there exist a (necessarily non-empty) $\mathscr A(H)$-word $\mathfrak{c} = c_1 \ast \cdots \ast c_m$ and an injection $\sigma: \llb 1, m \rrb \to \llb 1, n \rrb$ such that $\pi_H(\mathfrak c) = \pi_H(\mathfrak a) = x$ and $c_i \simeq_H a_{\sigma(i)}$ for every $i \in \llb 1, m \rrb$. Then
	\[
	\pi_K(\varphi^\ast(\mathfrak c)) = \varphi(c_1) \cdots \varphi(c_m) = \varphi(x)
	\quad\text{and}\quad
	\varphi(c_1) \simeq_K \varphi(a_{\sigma(1)}), \ldots, \varphi(c_m) \simeq_K \varphi(a_{\sigma(m)})
	\]
	(recall that monoid hom\-o\-mor\-phisms map units to units; so, if $u \simeq_H v$, then $\varphi(u) \simeq_K \varphi(v)$); and together with Proposition \ref{prop:min-basics}\ref{it:prop:min-basics(iib)}, this proves that 
	$\varphi^\ast(\mathfrak c) \prec_K \mathfrak b$, contradicting the $\preceq_K$-minimality of $\mathfrak b$.
	
	\ref{it:prop:min-equi(ii)} Fix $x \in H \setminus H^\times$, and suppose $\mathsf{L}_K^\m(\varphi(x)) \ne \emptyset$ (otherwise there is nothing to prove). Accordingly, let $k \in \mathsf{L}_K^\m(\varphi(x))$ and $\mathfrak b \in \mathcal Z_K^\m(\varphi(x))$ such that $k = \|\mathfrak b\|_K$. It is sufficient to check that $k \in \mathsf L_H^\m(x)$, and this is straightforward: Indeed, we have by \ref{it:prop:min-equi(i)} that $\varphi^\ast(\mathfrak a)$ is $\mathscr C_K$-congruent to $\mathfrak b$ for some $\mathfrak a \in \mathcal Z_H^\m(x)$, which implies in particular that $k = \|\varphi^\ast(\mathfrak a)\|_K = \|\mathfrak a\|_H \in \mathsf L_H^\m(x)$.
	
	\ref{it:prop:min-equi(iii)} Assume $\varphi$ is essentially surjective, and let $y \in K \setminus K^\times$. Then $y = u \varphi(x) v$ for some $u,v\in K^\times$ and $x\in H$, and neither $x$ is a unit of $H$ nor $\varphi(x)$ is a unit of $K$ (because $\varphi(H^\times) \subseteq K^\times$ and $y \notin K^\times$). Accordingly, we have by Lemma \ref{lem:min-unit-adjust} and part \ref{it:prop:min-equi(ii)} that $\mathsf{L}_K^\m(y) = \mathsf{L}_K^\m(\varphi(x)) \subseteq \mathsf{L}_H^\m(x)$.
\end{proof}

\subsection{Minimal factorizations in power monoids.}
\label{subsec:min-factorization-in-PMs}
Let $H$ be a monoid. Similarly as in \S{ }\ref{sec:atomicity}, we would like to simplify the study of minimal factorizations in $\PPc_\funt(H)$ as much as possible by passing to consideration of the reduced monoid $\PPc_\fun(H)$. For, it is of primary importance to make clear the nature of the relationship between minimal factorizations in $\PPc_\funt(H)$ and those in $\PPc_\fun(H)$.
We shall see that this is possible under \emph{some} circumstances. 
\begin{proposition}\label{prop:comm-pm}
	Let $H$ be a commutative monoid, and let $X\in \PPc_{\fin,1}(H)$. The following hold:
	\begin{enumerate}[label = {\rm (\roman{*})}]
		\item\label{it:prop:comm-pm(i)} $\ZZc_{\PPc_{\fin,1}(H)}^\m(X) \subseteq \ZZc_{\PPc_{\funt}(H)}^\m(X)$.
		\item\label{it:prop:comm-pm(ii)} $\mathsf{L}_{\PPc_{\fin,1}(H)}^\m(X) = \mathsf{L}_{\PPc_\funt(H)}^\m(X)$.
		\item\label{it:prop:comm-pm(iii)} $\LLc^\m(\PPc_{\fin,1}(H)) = \LLc^\m(\PPc_\funt(H))$.
	\end{enumerate}
\end{proposition}
\begin{proof}
	\ref{it:prop:comm-pm(i)} Let $\mathfrak{a}$ be a minimal factorization of $X$ relative to $\PPc_{\fin,1}(H)$. In light of Proposition \ref{prop:min-basics}\ref{it:prop:min-basics(i)}, $\mathfrak{a}$ is a non-empty $\mathscr A(\mathcal P_{\fin,1}(H))$-word, i.e., $\mathfrak a = A_1 \ast \cdots \ast A_n$ for some atoms $A_1, \ldots, A_n \in \mathcal P_{\fin,1}(H)$. 
	
	Assume for the sake of contradiction that $\mathfrak a$ is not a minimal factorization relative to $\PPc_{\fin,\times}(H)$. Then there exist a non-empty $\mathscr A(\mathcal P_{\fin,\times}(H))$-word $\mathfrak b = B_1 * \cdots * B_m$ and an injection $\sigma: \llb 1, m \rrb \to \llb 1, n \rrb$ with 
	\[
	X = A_1 \cdots A_n = B_1 \cdots B_m
	\quad\text{and}\quad
	B_1 \simeq_{\mathcal P_{\fin,\times}(H)} A_{\sigma(1)}, \ldots, B_m \simeq_{\mathcal P_{\fin,\times}(H)} A_{\sigma(m)},
	\]
	and on account of Proposition \ref{preorder-facts}\ref{it:prop:preorder-facts(iv)} we must have $1 \le m < n$.
	Since $H$ is a commutative monoid, this means in particular that, for each $i \in \llb 1, m \rrb$, there is $u_i \in H^\times$ such that $B_i = u_i A_{\sigma(i)}$. Thus we have
	\[
	A_1 \cdots A_n = B_1 \cdots B_m = (u_1 A_{\sigma(1)}) \cdots (u_m A_{\sigma(m)}) = u \cdot A_{\sigma(1)} \cdots A_{\sigma(m)},
	\]
	where $u := u_1 \cdots u_m \in H^\times$. In view of Proposition \ref{prop:pm-arith}\ref{it:prop:pm-arith(ii)}, it follows that
	\[
	\left| A_1 \cdots A_n \right| = \left|A_{\sigma(1)} \cdots A_{\sigma(m)} \right|,
	\]
	which is only possible if 
	\[
	X = A_1 \cdots A_n = A_{\sigma(1)} \cdots A_{\sigma(m)}, 
	\]
	because $1_H \in A_i$ for every $i \in \llb 1, n \rrb$, and hence $A_{\sigma(1)} \cdots A_{\sigma(m)} \subseteq A_1 \cdots A_n$ (note that here we use again that $H$ is commutative). So, letting $\mathfrak a^\prime$ be the $\mathscr A(\mathcal P_{\fin,1}(H))$-word $ A_{\sigma(1)} \ast \cdots \ast A_{\sigma(m)}$ and recalling from the above that $m \le n-1$, we see by Proposition \ref{preorder-facts}\ref{it:prop:preorder-facts(iv)} that $\mathfrak a^\prime \prec_{\mathcal P_{\fin,1}(H)} \mathfrak a$, which contradicts the hypothesis that $\mathfrak a$ is a minimal factorization of $X$ in $\mathcal P_{\fin,1}(H)$.
	
	\ref{it:prop:comm-pm(ii)} It is an immediate consequence of part \ref{it:prop:comm-pm(i)} and Propositions \ref{prop:funt&fun-have-the-same-system-of-lengths}\ref{it:prop:funt&fun-have-the-same-system-of-lengths(i)} and \ref{prop:min-equi}\ref{it:prop:min-equi(ii)}, when considering that every commutative monoid is Dedekind-finite.
	
	\ref{it:prop:comm-pm(iii)} We already know from part \ref{it:prop:comm-pm(ii)} that $\mathscr L^\m(\mathcal P_{\fin,1}(H)) \subseteq \mathscr L^\m(\mathcal P_\funt(H))$. For the opposite inclusion, fix $X \in \PPc_\funt(H)$. We claim that there exists $Y \in \mathcal P_{\fin,1}(H)$ with 
	$\mathsf L_{\PPc_\funt(H)}^\m(X) = \mathsf L_{\PPc_{\fin,1}(H)}^\m(Y)$.
Indeed, pick $x \in X \cap H^\times$. Then $x^{-1} X\in \PPc_{\fin,1}(H)$, and we derive from Lemma \ref{lem:min-unit-adjust} and part \ref{it:prop:comm-pm(ii)} that 
	\[
	\mathsf{L}_{\PPc_\funt(H)}^\m(X) = \mathsf{L}_{\PPc_\funt(H)}^\m(x^{-1}X) = \mathsf{L}_{\PPc_{\fin,1}(H)}^\m(x^{-1} X),
	\]
	which proves our claim and suffices to finish the proof (since $X$ was arbitrary).
\end{proof}
We will now discuss an instance in which equality in Proposition \ref{prop:comm-pm}\ref{it:prop:comm-pm(ii)} does not necessarily hold true in the absence of commutativity, and the best we can hope for is the containment relation implied by Proposition \ref{prop:min-equi}\ref{it:prop:min-equi(ii)} when $\varphi$ is the natural embedding of  Proposition \ref{prop:funt&fun-have-the-same-system-of-lengths}\ref{it:prop:funt&fun-have-the-same-system-of-lengths(i)}.
\begin{example}\label{exa:strict-inclusion}
	Let $n$ be a (positive) multiple of $105$, and $p$ a (positive) prime dividing $n^2 + n + 1$; note that $p \ge 11$ and $3 \le n \bmod p \le p-3$ (where $n \bmod p$ is the remainder of the Euclidean division of $n$ by $p$). Following \cite[p. 27]{Gor1980}, we take $H$ to be the metacyclic group generated by the $2$-element set $\{r, s\}$ subject to $\ord_H(r) = p$, $\ord_H(s) = 3$, and $s^{-1} r s = r^n$. 
	Then $H$ is a non-abelian group of odd order $3p$, and by Theorem \ref{th:atomicity} and Propositions \ref{prop:equimorphism}\ref{it:prop:equimorphism(ii)} and \ref{prop:funt&fun-have-the-same-system-of-lengths}\ref{it:prop:funt&fun-have-the-same-system-of-lengths(i)}, $\mathcal P_{\fin,1}(H)$ and $\mathcal P_{\fin,\times}(H)$ are atomic monoids. 
	
We claim that $X := \langle r \rangle_H$ has minimal factorizations of length $p-1$ in $\mathcal P_{\fin,1}(H)$ but not in $\mathcal P_{\fin,\times}(H)$.
	{Pick $g \in X \setminus \{1_H\}$.} Clearly $\ord_H(g) = p$, and thus we get from Lemma \ref{lem:2-elt-atoms}\ref{it:lem:2-elt-atoms(i)} that $\{1_H, g\}$ is an atom of $\mathcal P_{\fin,1}(H)$. Then it is immediate to see that $\mathfrak a_g := \{1_H, g\}^{\ast (p-1)}$ is a minimal factorization of $X$ in $\mathcal P_{\fin,1}(H)$; most notably, $\mathfrak a_g$ is minimal since otherwise there should exist an exponent $k \in \llb 1, p-2 \rrb$ such that $g^{p-1} = g^k$, contradicting that $\ord_H(g) = p$. Yet, $\mathfrak a_g$ is not a minimal factorization of $X$ in $\mathcal P_{\fin,\times}(H)$.
	Indeed, Proposition \ref{prop:funt&fun-have-the-same-system-of-lengths}\ref{it:prop:funt&fun-have-the-same-system-of-lengths(ii)} and Lemma \ref{lem:2-elt-atoms}\ref{it:lem:2-elt-atoms(i)} guarantee that $\{1_H, g\}$ and $\{1_H, g^n\}$ are associate atoms of $\mathcal P_{\fin,\times}(H)$, because $s^{-1} g^n s = g$ and, hence, $s^{-1} \{1, g\} s = \{1_H, g^n\}$. So, in view of Proposition \ref{preorder-facts}\ref{it:prop:preorder-facts(iv)}, it is straightforward that 
	\[
	\{1_H, g\}^{\ast (p-2)} \ast \{1_H, g^n\} \prec_{\mathcal P_{\fin,\times}(H)} \mathfrak a_g,
	\]
	In particular, note here that we have used that $3 \le n \bmod p \le p-3$ to obtain
	\[
	\{1_H, g, \ldots, g^{p-2}\} \cup \{g^n, g^{n+1}, \ldots, g^{n+p-2}\} = \{1_H, g, \ldots, g^{p-1}\} = X.
	\]
	Given that, suppose for a contradiction that $X$ has a minimal factorization $\mathfrak c$ of length $p-1$ in $\mathcal P_{\fin,\times}(H)$. Then by Propositions \ref{prop:funt&fun-have-the-same-system-of-lengths}\ref{it:prop:funt&fun-have-the-same-system-of-lengths(i)} and \ref{prop:min-equi}\ref{it:prop:min-equi(i)}, $\mathfrak c$ is $\mathscr C_{\mathcal P_{\fin,\times}(H)}$-congruent to a $\preceq_{\mathcal P_{\fin,1}(H)}$-minimal fac\-tor\-ization $\mathfrak a = A_1 \ast \cdots \ast A_{p-1}$ of $X$ of length $p-1$; and we aim to show that $\mathfrak a$ is $\mathscr C_{\mathcal P_{\fin,1}(H)}$-congruent to $\mathfrak a_g$ for some $g \in X \setminus \{1_H\}$, which is however impossible as it would mean that $\mathfrak a_g$ is a minimal factorization of $X$ in $\mathcal P_{\fin,\times}(H)$, in contradiction to what established in the above.
	
	Indeed, let $B_i$ be, for $i \in \llb 1, p-1 \rrb$, the image of $\{k \in \llb 0, p-1 \rrb: r^k \in A_i\} \subseteq \bf Z$ under the canonical map $\mathbf Z \to \mathbf Z/p\mathbf Z$. Then $\mathfrak a$ is a minimal factorization of $X$ in $\mathcal P_{\fin,1}(H)$ only if $\mathfrak b := B_1 \ast \cdots \ast B_{p-1}$ is a minimal factorization of $\mathbf Z/p\mathbf Z$ in the reduced power monoid of $(\mathbf Z/p\mathbf Z, +)$, herein denoted by $\mathcal P_{\fin,0}(\mathbf Z/p\mathbf Z)$.
	
	We want to show that $\mathfrak b$ is $\preceq_{\mathcal P_{\fin,0}(H)}$-minimal only if there is a non-zero $x \in \mathbf Z/p\mathbf Z$ such that $B_i = \{\overline{0}, x\}$ or $B_i = \{\overline{0}, -x\}$, or equivalently $A_i = \{1_H, r^{\hat{x}}\}$ or $A_i = \{1_H, r^{-\hat{x}}\}$, for every $i \in \llb 1, p-1 \rrb$ (for notation, see \S{ }\ref{subsec:generalities}). By the preceding arguments, this will suffice to conclude that $p-1 \notin \mathsf L^\m_{\mathcal P_\funt(H)}(X)$, because it implies at once that $\mathfrak a$ is $\mathcal C_{\mathcal P_\fun(H)}$-congruent to $\mathfrak a_g$ with $g := r^{\hat{x}} \in X \setminus \{1_H\}$.
	
	To begin, let $K$ be a subset of $\llb 1, p-1 \rrb$, and define $\mathcal{S}_K := \sum_{k \in K} B_k$ and $s_K := \{k \in K: |B_k| \ge 3\}$. Then we have by the Cauchy-Davenport inequality (see, e.g., \cite[Theorem 6.2]{Gry13}) that
	\begin{equation}\label{equ:cauchy-davenport-application}
	\mathcal{S}_K = \mathbf Z/p\mathbf Z 
	\quad\text{or}\quad
	|\mathcal{S}_K| \ge 1 + {\sum}_{k \in K} \bigl(|B_k| - 1\bigr) \ge 1 + |K| + s_K.
	\end{equation}
	Now, let $I$ and $J$ be disjoint subsets of $\llb 1, p-1 \rrb$ with $|I \cup J| = |I| + |J| = p-2$. We claim $s_I = s_J = 0$. Indeed, it is clear that $\mathcal{S}_{I \cup J} \ne \mathbf Z/p\mathbf Z$, otherwise $\mathfrak b$ would not be a minimal factorization in $\mathcal P_{\fin,0}(\mathbf Z/p\mathbf Z)$. So, another application of the Cauchy-Davenport inequality, combined with \eqref{equ:cauchy-davenport-application}, yields
	\begin{equation}\label{equ:another-cauchy-davenport-application}
	|S_{I \cup J}| = |S_I + S_J| \ge |S_I| + |S_J| - 1 \ge 1 + |I| + |J| + s_I + s_J = p-1 + s_I + s_J.
	\end{equation}
	This suffices to prove that $|S_I + S_J| = p-1$ and $s_I = s_J = 0$, or else $S_{I \cup J} = \mathbf Z/p\mathbf Z$ (a contradiction). 
	
	It follows $|B_1| = \cdots = |B_{p-1}| = 2$. So, taking $I$ in \eqref{equ:another-cauchy-davenport-application} to range over all $1$-element subsets of $\llb 1, p-1 \rrb$ and observing that, consequently, $|S_J| \ge p-1-|S_I| = p-3 \ge 8 > |S_I|$, we infer from Vosper's theorem (see, e.g., \cite[Theorem 8.1]{Gry13}) that there exists a non-zero $x \in \mathbf Z/p\mathbf Z$ such that, for every $i \in \llb 1, p-1 \rrb$, $B_i$ is an arithmetic progression of $\mathbf Z/p\mathbf Z$ with difference $x$, i.e., $B_i = \{\overline{0}, x\}$ or $B_i = \{\overline{0}, -x\}$ (as wished).
\end{example}

We proceed with an analogue of Theorem \ref{thm:BF-torsion}\ref{it:thm:BF-torsion(i)} and then prove the main results of the section.

\begin{proposition}\label{prop:bounded-minimal-fzn}
	Let $H$ be a monoid and $X \in \PPc_\funt(H)$. The following hold:
	\begin{enumerate}[label={\rm (\roman{*})}]
		\item\label{it:prop:bounded-minimal-fzn(i)} If $X\in \PPc_\fun(H)$, then a minimal factorization of $X$ in $\PPc_\fun(H)$ has length $\le |X|-1$.
		\item\label{it:prop:bounded-minimal-fzn(ii)} If $H$ is Dedekind-finite, then a minimal factorization of $X$ in $\PPc_\funt(H)$ has length $\le |X|-1$.
	\end{enumerate}
\end{proposition}
\begin{proof}
	 \ref{it:prop:bounded-minimal-fzn(i)} The claim is trivial if $X = \{1_H\}$, when the only factorization of $X$ is the empty word; or if $X \in \mathscr A(\PPc_{\fin,1}(H))$, in which case $|X| \ge 2$ and $X$ has a unique factorization (of length $1$). So, assume that $X$ is neither the identity nor an atom of $\mathcal P_{\fin,1}(H)$, and let $\mathfrak a$ be a minimal factorization of $X$ (relative to $\PPc_{\fin,1}(H)$). Then $\mathfrak a = A_1*\cdots* A_n$, where $A_1, \ldots, A_n \in \mathscr A(\mathcal P_{\fin,1}(H))$ and $n \ge 2$; and we claim that
	 \[
	 A_1\cdots A_i \subsetneq A_1\cdots A_{i+1}, \quad \text{for every }i \in \llb 1, n-1 \rrb.
	 \]
	 In fact, let $i \in \llb 1, n-1 \rrb$. Since $1_H \in A_{i+1}$, it is clear that $A_1\cdots A_i \subsetneq A_1\cdots A_{i+1}$; and the inclusion must be strict, or else $A_1 \ast \cdots \ast A_i \ast \mathfrak b \prec_{\mathcal P_{\fin,1}(H)} \mathfrak a$, where $\mathfrak b := \varepsilon_{\mathscr A(\mathcal P_{\fin,1}(H))}$ if $i = n-1$ and $\mathfrak b := A_{i+2} \ast \cdots \ast A_n$ otherwise (contradicting the minimality of $\mathfrak a$). Consequently, we see that
	 $
	 2 \le |A_1\cdots A_i | < |A_1\cdots A_{i+1}| \le |X|$ for all $i \in \llb 1, n-1 \rrb$, and this implies at once that $n\le |X|-1$.
	
	\ref{it:prop:bounded-minimal-fzn(ii)} The conclusion is immediate from part \ref{it:prop:bounded-minimal-fzn(i)} and Propositions \ref{prop:funt&fun-have-the-same-system-of-lengths}\ref{it:prop:funt&fun-have-the-same-system-of-lengths(i)} and \ref{prop:min-equi}\ref{it:prop:min-equi(iii)}.
\end{proof}
\begin{theorem}\label{BmF-char}
	Let $H$ be a monoid. Then the following are equivalent:
	\begin{enumerate}[label={\rm (\alph{*})}]
		\item\label{it:BmF-char(a)} $1_H \ne x^2 \ne x$ for every $x \in H \setminus \{1_H\}$.
		\item\label{it:BmF-char(b)} $\PPc_\fun(H)$ is atomic.
		\item\label{it:BmF-char(c)} $\PPc_\fun(H)$ is \textup{BmF}.
		\item\label{it:BmF-char(d)} $\PPc_\fun(H)$ is \textup{FmF}.
		\item\label{it:BmF-char(e)} Every $2$-element subset $X$ of $H$ with $1_H \in X$ is an atom of $\mathcal P_\fun(H)$.
		\item\label{it:BmF-char(f)} $\PPc_\funt(H)$ is atomic.
		\item\label{it:BmF-char(g)} $\PPc_\funt(H)$ is \textup{BmF}.
		\item\label{it:BmF-char(h)} $\PPc_\funt(H)$ is \textup{FmF}.
		\item\label{it:BmF-char(i)} Every $2$-element subset $X$ of $H$ with $X \cap H^\times \ne \emptyset$ is an atom of $\mathcal P_\funt(H)$.
	\end{enumerate}
\end{theorem}
\begin{proof}
	We already know from Theorem \ref{th:atomicity} and Lemma \ref{lem:2-elt-atoms} that \ref{it:BmF-char(b)} $\Leftrightarrow$ \ref{it:BmF-char(a)} $\Leftrightarrow$ \ref{it:BmF-char(e)} and \ref{it:BmF-char(i)} $\Rightarrow$ \ref{it:BmF-char(a)}; while it is straightforward from our definitions that \ref{it:BmF-char(h)} $\Rightarrow$ \ref{it:BmF-char(g)} $\Rightarrow$ \ref{it:BmF-char(f)}. So, it will suffice to prove that \ref{it:BmF-char(b)}
	$\Rightarrow$ \ref{it:BmF-char(c)} $\Rightarrow$ \ref{it:BmF-char(d)} $\Rightarrow$ \ref{it:BmF-char(h)}
 and \ref{it:BmF-char(f)} $\Rightarrow$
	\ref{it:BmF-char(i)}.
	
	\ref{it:BmF-char(b)} $\Rightarrow$ \ref{it:BmF-char(c)}: If $X \in \mathcal P_\fun(H)$ is a non-unit, then $\mathcal{Z}_{\mathcal P_\fun(H)}(X)$ is non-empty, and by Propositions \ref{prop:min-basics}\ref{it:prop:min-basics(ii)} and \ref{prop:bounded-minimal-fzn}\ref{it:prop:bounded-minimal-fzn(i)} we have that $\emptyset \ne \mathsf{L}_{\mathcal P_\fun(H)}^\m(X) \subseteq \llb 1, |X|-1\rrb$. So, $\mathcal P_\fun(H)$ is BmF.
	
	\ref{it:BmF-char(c)} $\Rightarrow$ \ref{it:BmF-char(d)}: Let $X \in \mathcal P_\fun(H)$ be a non-unit. 
	By Proposition \ref{prop:pm-arith}\ref{it:prop:pm-arith(i)}, any atom of $\mathcal P_\fun(H)$ dividing $X$ must be a subset of $X$, and there are only finitely many of these (since $X$ is finite).
	Because a minimal factorization of $X$ is a bounded $\mathscr A(\mathcal P_\fun(H))$-word (by the assumption that $H$ is BmF), it follows that $X$ has finitely many minimal factorizations, and hence $\mathcal P_\fun(H)$ is FmF (since $X$ was arbitrary).
	
	\ref{it:BmF-char(d)} $\Rightarrow$ \ref{it:BmF-char(h)}: Pick a non-unit $X \in \mathcal P_{\fin,\times}(H)$, and let $u \in H^\times$ such that $uX \in \mathcal P_{\fin,1}(H)$. Since $\mathcal P_{\fin,1}(H)$ is FmF (by hypothesis), it is also atomic. Hence, by Theorem \ref{th:atomicity} and Lemma \ref{lem:no-non-id-elts-of-small-order-implies-structure}\ref{it:lem:no-non-id-elts-of-small-order-implies-structure(i)}, $H$ is Dedekind-finite, and so we have by Proposition \ref{prop:funt&fun-have-the-same-system-of-lengths}\ref{it:prop:funt&fun-have-the-same-system-of-lengths(i)} that the natural embedding $\mathcal P_{\fin,1}(H) \hookrightarrow \mathcal P_{\fin,\times}(H)$ is an essentially surjective equimorphism.  
	In particular, we infer from Proposition \ref{prop:min-equi}\ref{it:prop:min-equi(i)} that any minimal factorization of $uX$ in $\mathcal P_{\fin,\times}(H)$ is $\mathscr C_{\mathcal P_{\fin,\times}(H)}$-congruent to a minimal factorization of $uX$ in $\mathcal P_{\fin,1}(H)$.
	However, this makes $\mathsf{Z}_{\mathcal P_{\fin,\times}(H)}^\m(uX)$ finite, whence $\mathsf{Z}_{\mathcal P_{\fin,\times}(H)}^\m(X)$ must also be finite as a consequence of Lemma \ref{lem:min-unit-adjust}.

	\ref{it:BmF-char(f)} $\Rightarrow$ \ref{it:BmF-char(i)}: Let $X$ be a $2$-element subset of $H$ with $X \cap H^\times \ne \emptyset$. Then $X = uA$ for some unit $u \in H^\times$, where $A := u^{-1}X$ is a $2$-element subset of $H$ with $1_H \in H$; and since $\mathcal P_{\fin,\times}(H)$ is atomic (by hypothesis), we are guaranteed by Lemmas \ref{lem:2-elt-atoms} and \ref{lem:no-non-id-elts-of-small-order-implies-structure}\ref{it:lem:no-non-id-elts-of-small-order-implies-structure(i)} that $A$ is an atom of $\mathcal P_\fun(H)$ and $H$ is Dedekind-finite. Therefore, we conclude from Proposition \ref{prop:funt&fun-have-the-same-system-of-lengths}\ref{it:prop:funt&fun-have-the-same-system-of-lengths(ii)} that $X \in \mathscr A(\mathcal P_\funt(H))$.
\end{proof}

\begin{theorem}\label{prop:HF-exp-3}
	Let $H$ be a monoid. Then $\mathcal P_\fun(H)$ is \textup{HmF} if and only if $H$ is trivial or a cyclic group of order $3$.
\end{theorem}

\begin{proof}
	The ``if'' part is an easy consequence of Theorem \ref{BmF-char} and Propositions \ref{prop:bounded-minimal-fzn}\ref{it:prop:bounded-minimal-fzn(i)} and \ref{prop:min-basics}\ref{it:prop:min-basics(i)}, when considering that, if $H$ is trivial or a cyclic group of order $3$, then $1_H \ne x^2 \ne x$ for all $x \in H \setminus \{1_H\}$ and every non-empty subset of $H$ has at most $3$ elements.

	As for the other direction, suppose $\mathcal P_\fun(H)$ is HmF and $H$ is non-trivial. Then $\mathcal P_\fun(H)$ is atomic, and we claim that $H$ is a $3$-group. By Theorem \ref{th:atomicity} and Lemma \ref{lem:no-non-id-elts-of-small-order-implies-structure}\ref{it:lem:no-non-id-elts-of-small-order-implies-structure(ii)}, it suffices to show that
	$
	x^3 \in \{1_H, x, x^2\}$ for every $x \in H$, since this in turn implies (by induction) that $\langle x \rangle_H \subseteq \{1_H, x, x^2\}$ and $\ord_H(x) \le 3$.
	
	{Assume to the contrary} that $x^3 \notin \{1_H, x, x^3\}$ for some $x \in H$, and set $X := \{1_H, x, x^2, x^3\}$. By Theorem \ref{BmF-char}, $\mathfrak a := \{1_H, x\}^{\ast 3}$ and $\mathfrak b := \{1_H, x\} \ast \{1_H, x^2\}$ are both factorizations of $X$ in $\mathcal P_\fun(H)$; and in light of Proposition \ref{prop:min-basics}\ref{it:prop:min-basics(i)}, $\mathfrak b$ is in fact a minimal factorization (of length $2$). Then $\mathfrak a$ cannot be minimal, because $\mathcal P_\fun(H)$ is HmF and $\mathfrak a$ has length $3$. However, since $\mathcal P_\fun(H)$ is a reduced monoid (and $X$ is not an atom), this is only possible if $x^3 \in X = \{1_H, x\}^2$, a contradiction.
	
	So, $H$ is a $3$-group, and as such it has a non-trivial center $Z(H)$, see e.g. \cite[Theorem 2.11(i)]{Gor1980}. Let $z$ be an element in $Z(H) \setminus \{1_H\}$, and suppose for a contradiction that $H$ is not cyclic. Then we can choose some element $y \in H \setminus \gen{z}_H$, and it follows from the above that $K := \gen{y,z}_H$ is an abelian subgroup of $H$ with $\ord_H(y) = \ord_H(z) = 3$ and $|K| = 9$. 
	We will prove that $K$ has $\preceq_{\mathcal P_\fun(H)}$-minimal factorizations of more than one length, which is a contradiction and finishes the proof.
	
	Indeed, we are guaranteed by Theorem \ref{BmF-char} that $\mathfrak c := \{1_H,y\}^{\ast 2} \ast \{1_H,z\}^{\ast 2}$ is a length-$4$ factorization of $K$ in $\PPc_\fun(H)$; and it is actually a minimal factorization, because removing one or more atoms from $\mathfrak c$ yields an $\mathscr A(\mathcal P_\fun(H))$-word whose image under $\pi_{\mathcal P_\fun(H)}$ has cardinality at most $8$ (whereas we have already noted that $|K| = 9$).
	On the other hand, it is not difficult to check that $A := \{1_H, y, z\}$ is an atom of $\mathcal P_\fun(H)$: If $\{1_H, y, z\} = YZ$ for some $Y, Z \in \mathcal P_\fun(H)$ with $|Y|, |Z| \ge 2$, then $Y, Z \subseteq \{1_H, y, z\}$ and $Y \cap Z = \{1_H\}$, whence $YZ = \{1_H, y\}\cdot\{1_H, z\} = K \ne A$. This in turn implies that $A^{\ast 2}$ is a length-$2$ factorization of $K$ in $\PPc_\fun(H)$, and it is minimal by Proposition \ref{prop:min-basics}\ref{it:prop:min-basics(i)}.
	So, we are done.
\end{proof}

\begin{corollary}\label{cor:when-reduced-pm-is-minimally-factorial}
	Let $H$ be a monoid. Then $\mathcal P_\fun(H)$ is minimally factorial if and only if $H$ is trivial.
\end{corollary}

\begin{proof}
	The ``if'' part is obvious. For the other direction, assume by way of contradiction that $\mathcal P_\fun(H)$ is minimally factorial but $H$ is non-trivial. Then $\mathcal P_\fun(H)$ is HmF, and we obtain from Theorem \ref{prop:HF-exp-3} that $H$ is a cyclic group of order $3$. Accordingly, let $x$ be a generator of $H$. By Lemma \ref{lem:2-elt-atoms}\ref{it:lem:2-elt-atoms(i)} and Proposition \ref{prop:min-basics}\ref{it:prop:min-basics(i)}, $\mathfrak a := \{1_H, x\}^{\ast 2}$ and $\mathfrak b := \{1_H, x^2\}^{\ast 2}$ are both minimal factorizations of $H$ in $\mathcal P_{\fin,1}(H)$. However, $(\mathfrak a, \mathfrak b) \notin \mathscr C_{\mathcal P_\fun}(H)$, because $\mathcal P_\fun(H)$ is a reduced monoid. Therefore, $\mathcal P_\fun(H)$ is not minimally factorial, so leading to a contradiction and completing the proof.
\end{proof}
At this point, we have completely characterized the correlation between the ground monoid $H$ and whether $\PPc_\fun(H)$ has factorization properties such as atomicity, BFness, etc., and their minimal counterparts.
In most cases, this extends to a characterization of whether the same properties hold for $\PPc_\funt(H)$, with the exception of the gap suggested by Theorem \ref{prop:HF-exp-3} and Corollary \ref{cor:when-reduced-pm-is-minimally-factorial}.
In particular, it still remains to determine the monoids $H$ which make $\PPc_\funt(H)$ HmF or minimally factorial.
However, what we have shown indicates, we believe, that the arithmetic of $\PPc_\fun(H)$ and $\PPc_\funt(H)$ is robust and ripe for more focused study.

\section{Cyclic monoids and interval length sets}
\label{sec:cyclic-case} 
For those monoids $H$ with $\PPc_\fun(H)$ atomic, we have by Proposition \ref{lem:no-non-id-elts-of-small-order-implies-structure} that the semigroup generated by an element $x\in H$ is isomorphic either to $\ZZb/n\ZZb$ or to $\NNb$ under addition.
As such, we will concentrate throughout on factorizations in $\PPc_{\fin,0}(\mathbf{Z}/n\mathbf{Z})$ and also mention some results on $\PPc_{\fin,0}(\NNb)$ which are discussed in detail in \cite[\S{ }4]{FTr17}.
At the end we will return to the general case, where the preceding discussion will culminate in a realization result (Theorem \ref{th:interval-lengths}) for sets of minimal lengths of $\PPc_\fun(H)$.

We invite the reader to review \S{ }\ref{subsec:generalities} before reading further. Also, note that, through the whole section, we have replaced the notation $\PPc_\fun(H)$ with $\PPc_{\fin,0}(H)$ when $H$ is written additively (cf. Example \ref{exa:strict-inclusion}).

\begin{definition}\label{NR-factorization}
	Let $X \in \mathcal P_{\fin,0}(\ZZb/n\ZZb)$. We say that a non-empty factorization $\mathfrak a = A_1 \ast \cdots \ast A_\ell \in \mathcal{Z}(X)$ is a \emph{non-reducible factorization} (or, shortly, an \emph{\textup{NR}-factorization}) if $ \max\hat{A}_1 + \dots + \max\hat{A}_\ell = \max \hat{X}$.
\end{definition}

This condition on factorizations will allow us to bring calculations up to the integers, where sumsets are more easily understood.
More importantly, NR-factorizations are very immediately relevant to our investigation of minimal factorizations.

\begin{lemma}\label{NR-factorizations-are-minimal}
Any \textup{NR}-factorization in $\PPc_{\fin,0}(\mathbf{Z}/n\mathbf{Z})$ is a minimal factorization.
\end{lemma}

\begin{proof}
Let $\mathfrak{a} = A_1 \ast \cdots \ast A_\ell$ be an NR-factorization in $\mathcal P_{\fin,0}(\mathbf Z/n\mathbf Z)$ of length $\ell$, and assume for the sake of contradiction that $\mathfrak a$ is not minimal. Since
$\PPc_{\fin,0}(\mathbf{Z}/n\mathbf{Z})$ is reduced and commutative, the factorizations which are $\mathscr C_{\PPc_{\fin,0}(\mathbf{Z}/n\mathbf{Z})}$-congruent to $\mathfrak{a}$ are exactly the words $A_{\sigma(1)} \ast \cdots \ast A_{\sigma(\ell)}$, where $\sigma$ is an arbitrary permutation of the interval $\llb 1, \ell \rrb$.
So, on account of Proposition \ref{prop:min-basics}\ref{it:prop:min-basics(i)}, the non-minimality of $\mathfrak{a}$ implies without loss of generality that $\ell \ge 3$ and $
X := A_1 + \cdots + A_\ell = A_1 + \cdots + A_k$ for some $k \in \llb 1, \ell-1 \rrb$. 

Now, let $x \in X$ such that $\hat{x} = \max \hat{X}$. Using that $\mathfrak a$ is an NR-factorization, and considering that, for each $i \in \llb 1, \ell \rrb$, $A_i$ is an atom of $\mathcal P_{\fin,0}(\mathbf Z/n\mathbf Z)$ and hence $\max \hat{A}_i \ge 1$, it follows from the above that
\begin{equation}\label{equ:inequ-with-maxima}
\hat{x} = \max\hat{A}_1+\max\hat{A}_2+\cdots+\max\hat{A}_\ell > \max\hat{A}_1 + \dots + \max\hat{A}_k,
\end{equation}
On the other hand, since $X = A_1 + \cdots + A_k$, there are $a_1 \in A_1,\dots, a_k \in A_k$ such that $a_1+\dots+ a_k = x$, from which we see that $\hat{x} \equiv \hat{a}_1 + \cdots + \hat{a}_k \bmod n$. But it follows from \eqref{equ:inequ-with-maxima} that
$
0 \le \hat{a}_1+\dots+ \hat{a}_k < \hat{x} < n$, and this implies $\hat{x} \not\equiv \hat{a}_1 + \cdots + \hat{a}_k \bmod n$ (recall that, by definition, $\hat{X} \subseteq \llb 0, n-1 \rrb$). So we got a contradiction, showing that $\mathfrak a$ was minimal and completing the proof.
\end{proof}

We are aiming to find, for every $k\in\llb 2, n-1\rrb$, a set $X_k\in\PPc_{\fin,0}(\ZZb/n\ZZb)$ for which $\mathsf{L}^\m(X_k) = \llb 2, k \rrb$, on the assumption that $n\ge 5$ is odd:
Most of the difficulty lies in showing that $2\in \mathsf{L}^\m(X_k)$.
To do this, we first need to produce some large atoms.

\begin{proposition}\label{large-atom-construction}
Let $n\ge 5$ be odd.
Then the following sets are atoms of $\mathcal P_{\fin,0}(\mathbf Z/n\mathbf Z)$:
\begin{enumerate}[label={\rm (\roman{*})}]
\item\label{it:large-atom-construction(i)} $B_h := \bigl\{\overline{0}\} \cup \{\overline{1},\overline{3},\dots, \overline{h}\bigr\}$ for odd $h\in \llb 1,(n-1)/2 \rrb$.
\item\label{it:large-atom-construction(ii)} $C_1 := \bigl\{\overline{0}, \overline{2}\bigr\}$, $C_3 := \bigl\{\overline{0},\overline{2},\overline{3},\overline{4}\bigr\}$, and $C_\ell := B_\ell\cup\bigl\{\overline{\ell+1}\bigr\}$ for odd $\ell\in \llb 5, (n-1)/2 \rrb$.
\end{enumerate}
\end{proposition}

\begin{proof}
\ref{it:large-atom-construction(i)} Let $h \in \llb 1, (n-1)/2 \rrb$ be odd, and suppose that $B_h = X + Y$ for some $X,Y\in \mathcal P_{\fin,0}(\mathbf Z/n\mathbf Z)$.
Then $X$ and $Y$ are subsets of $B_h$, so
\[
\max\hat{X} + \max\hat{Y} \le 2\max\hat{B}_h = 2h \le n-1.
\]
Because $\overline{1}\in B_h$, we must have $\overline{1}\in X\cup Y$.
However, if $\overline{1}\in X$ and $a\in Y$ for some $a \in B_h\setminus\{\overline{0}\}$, then  $1+\hat{a} \in \hat{X} + \hat{Y}$ is even, which is impossible since $\max\hat{X} + \max\hat{Y} < n$ and $\hat{B}_h \setminus\{0\}$ consists only of odd numbers.
Thus $Y = \{\overline{0}\}$, and hence $B_h$ is an atom.

\ref{it:large-atom-construction(ii)}
$C_1$ is an atom by Lemma \ref{lem:2-elt-atoms}\ref{it:lem:2-elt-atoms(i)} and it is not too difficult to see that so is $C_3$. Therefore, let $\ell\ge 5$ and suppose $C_\ell = X+Y$ for some $X,Y\in \mathcal P_{\fin,0}(\mathbf Z/n\mathbf Z)$ with $X,Y\neq\bigl\{\overline{0}\bigr\}$.

First assume that $\overline{\ell+1}\notin X\cup Y$. Then $\hat{X}$ and $\hat{Y}$ consist only of odd integers, so $\hat{x}+\hat{y}$ is an even integer in the interval $\llb 2, n-1 \rrb$ for all $x\in X\setminus\bigl\{\overline{0}\bigr\}$ and $y\in Y\setminus\bigl\{\overline{0}\bigr\}$.
However, $\hat{X}+\hat{Y} = \hat{C}_\ell$ and the only non-zero even element of $\hat{C}_\ell$ is $\ell+1$. Thus, it must be that $X = \bigl\{\overline{0},x\bigr\}$ and $Y = \bigl\{\overline{0},y\bigr\}$ for some non-zero $x, y \in \mathbf Z/n\mathbf Z$, with the result that $|X+Y| \le 4 < |C_\ell|$, a contradiction.

It follows (without loss of generality) that $\overline{\ell+1}\in Y$.
Then $X \subseteq \bigl\{\overline{0},\overline{\ell},\overline{\ell+1}\bigr\}$, for, if $x\in X$ with $0 < \hat{x} < \ell$, then $\hat{x} + \ell+1 \in \hat{C}_\ell$, which is impossible since $\hat{x}+\ell+1 \in \llb \max\hat{C}_\ell+1 , n-1 \rrb$.
This in turn implies that $Y \subseteq \bigl\{\overline{0},\overline{1},\overline{\ell},\overline{\ell+1}\bigr\}$ for similar reasons. As a consequence,
\[
X+Y
\subseteq \bigl\{\overline{0},\overline{\ell},\overline{\ell+1}\bigr\}
+ \bigl\{\overline{0},\overline{1},\overline{\ell},\overline{\ell+1}\bigr\}
= \bigl\{ \overline{0}, \overline{1}, \overline{\ell}, \overline{\ell+1}, \overline{2\ell},\overline{2\ell+1}, \overline{2\ell+2} \bigr\}
\]
However, $\ell+1 < 2\ell \le n-1$, so we cannot have $\overline{2\ell}\in X+Y$. Then $2\ell+1 = n$, in which case $\overline{2\ell+1} = \overline{0}$ and $\overline{2\ell+2} = \overline{1}$; or $2\ell+1 < n$, so that $\overline{2\ell+1}, \overline{2\ell+2} \notin C_h$ (recall that $\ell \le (n-1)/2$). In either case, we get $X+Y \subseteq \bigl\{ \overline{0}, \overline{1}, \overline{\ell}, \overline{\ell+1}\bigr\}$, hence $|X+Y| \le 4 < |C_\ell|$, which is a contradiction and leads us to conclude that $C_\ell$ is an atom.
\end{proof}

Now that we have found large atoms in $\PPc_{\fin,0}(\ZZb/n\ZZb)$, we can explicitly give, for each $k\in \llb 2,n-1\rrb$, an element $X_k\in \PPc_{\fin,0}(\ZZb/n\ZZb)$ which has a (minimal) factorization of length $2$.

\begin{lemma}\label{2-atom-factorization}
Fix an odd integer $n\ge 5$ and let $k\in \llb 2, n-1 \rrb$.
Then the set $X_k = \{\overline{0},\overline{1},\dots, \overline{k}\}$ has an \textup{NR}-factorization into two atoms in $\PPc_{\fin,0}(\mathbf{Z}/n\mathbf{Z})$.
\end{lemma}

\begin{proof}
We will use the atoms $B_h$ and $C_\ell$ as defined in Proposition \ref{large-atom-construction}. We claim that, for every $r\in\{0,1\}$ and all odd $h\in \llb 1, (n-1)/2 \rrb$,
\[
\hat{B}_{h+2r}+\hat{C}_h = \llb 0, 2h+2r+1 \rrb
\quad\text{and}\quad
\hat{C}_{h+2r} + \hat{C}_{h} = \llb 0, 2r+2h+2\rrb.
\]
We will only demonstrate that $\hat{B}_h + \hat{C}_h = \llb 0, 2h+1 \rrb$ (the other cases are an easy consequence).
The claim is trivial if $h=1$ or $h=3$, so suppose $h \ge 5$.
Then
\[
\hat{B}_h + \hat{C}_h \supseteq \{1,3,\dots,h\} + \{0, h+1\} = \{1,3,\dots, 2h+1\}
\]
and
\[
\hat{B}_h + \hat{C}_h \supseteq \{1,3,\dots,h\} + \{1, h\} = \{2,4,\dots, 2h\},
\]
so $\hat{B}_h +\hat{C}_h \supseteq \llb 0 , 2h+1 \rrb$.
This gives that $\hat{B}_h +\hat{C}_h = \llb 0 , 2h+1 \rrb$, since $\max\hat{B}_h+\max\hat{C}_h = h+(h+1)$.

Accordingly, we now prove that $X_k$ can be expressed as a two-term sum involving $B_h$ and $C_\ell$, for some suitable choices of $h$ and $\ell$ depending on the parity of $k$.
\begin{enumerate}[leftmargin=1.8cm,label={\textsc{Case }\arabic{*}:}]
	\item $k = 2m+1$ (i.e., $k$ is odd). Then it is immediate to verify that $X_k = B_m+C_m$ if $m$ is odd, and $X_k = B_{m+1} +C_{m-1}$ if $m$ is even.
	\item $k = 2m$ (i.e., $k$ is even). Since $X_2 = B_1+B_1$ and $X_4 = B_1+B_3$, we may assume $m\ge3$. Then it is seen that $X_k = C_{m} + C_{m-2}$ if $m$ is odd, and $X_k = C_{m-1}+C_{m-1}$ if $m$ is even.
\end{enumerate}
We are left to show that the decompositions given above do in fact correspond to minimal factorizations. As an example, consider the case when $k=2m+1$ and $m$ is odd (the computation will be essentially identical in the other cases).
Then $\max\hat{B}_{m}+\max\hat{C}_m = 2m+1$, so that $B_m \ast C_m$ is an NR-factorization of $X_k$, and is hence minimal by Proposition \ref{NR-factorizations-are-minimal}.
\end{proof}

\begin{lemma}\label{lem:interval-minimal-length-sets}
Fix an odd integer $n\ge 3$ and, for each $k\in \llb2,n-1\rrb$, let $X_k := \{\overline{0},\overline{1}, \dots, \overline{k} \} \in  \PPc_{\fin,0}(\mathbf{Z}/n\mathbf{Z})$.
Then $\mathsf{L}^\m(X_k) = \llb 2,k \rrb$.
\end{lemma}

\begin{proof}
We have already established in Lemma \ref{2-atom-factorization} that $X_2$ has an NR-factorization of length $2$.
Now fix $k \in \llb 3, n-1 \rrb$ and suppose that, for all $h \in \llb 2, k-1 \rrb$ and $\ell \in \llb 2,h \rrb$, $X_h$ has an NR-factorization of length $\ell$.
Choose some $\ell \in \llb 2,k-1 \rrb$; $X_{k-1}$ has an NR-factorization $\mathfrak{a}$, and it is straightforward to see that $\{\overline{0},\overline{1}\}*\mathfrak{a}$ is an NR-factorization of $X_k$.
Letting $\ell$ range over $\llb 2,k-1\rrb$, this argument, Lemma \ref{NR-factorizations-are-minimal}, and Lemma \ref{2-atom-factorization} imply that $\mathsf{L}^\m(X_k) \supseteq \llb 2, k \rrb$.
Moreover, Proposition \ref{prop:bounded-minimal-fzn}\ref{it:prop:bounded-minimal-fzn(i)} yields the other inclusion and so we have  $\mathsf{L}^\m(X_k) = \llb 2,k \rrb$.
\end{proof}

\begin{lemma}\label{prop:intervals-in-N}
Let $H$ be a non-torsion monoid.
Then $\mathscr{L}(\PPc_{\fin,0}(\NNb)) \subseteq \mathscr{L}^\m(\PPc_\fun(H))$, and for every $k\ge 2$ there exists $Y_k\in \PPc_\fun(H)$ with $\mathsf{L}^\m(Y_k) = \llb 2, k \rrb$.
\end{lemma}

\begin{proof}
Suppose that $y\in H$ has infinite order, and set $Y := \{y^k: k \in \mathbf N\}$. Clearly, $Y$ is a submonoid of $H$, and the (monoid) homomorphism $(\NNb,+) \to Y: k \mapsto y^k$ determined by sending $1$ to $y$ induces an iso\-morphism $\PPc_{\fin,0}(\NNb) \to \PPc_\fun(Y)$.
Since, by Proposition \ref{prop:pm-arith}\ref{it:prop:pm-arith(iii)}, $\PPc_\fun(Y)$ is a divisor-closed submonoid of $\PPc_\fun(H)$, we thus have by parts \ref{it:prop:min-basics(iii)} and \ref{it:prop:min-basics(iv)} of Proposition \ref{prop:min-basics} that
\[
\mathscr{L}(\PPc_{\fin,0}(\NNb)) = \mathscr{L}^\m(\PPc_{\fin,0}(\NNb)) = \mathscr{L}^\m(\PPc_\fun(Y)) \subseteq \mathscr{L}^\m(\PPc_\fun(H)).
\]
The rest of the statement now follows from the above and \cite[Proposition 4.8]{FTr17}.
\end{proof}

\begin{theorem}\label{th:interval-lengths}
Assume $H$ is a monoid such that $1_H \neq x^2 \neq x$ for all $x\in H\setminus\{1_H\}$, and set $N := \sup\{\ord_H(x) : x\in H \}$.
Then $\llb 2, k \rrb \in \mathscr{L}^\m(\PPc_\fun(H))$ for every $k \in \llb 2, N-1 \rrb$.
\end{theorem}

\begin{proof}
If $H$ is non-torsion, this follows immediately from Lemma \ref{prop:intervals-in-N}.
Otherwise, let $k\in \llb 2, N-1 \rrb$ and $y\in H$ with $n := \ord_H(x) > k$.
Then $Y :=\gen{y}_H \cong \ZZb/n\ZZb$, so we have by Proposition \ref{prop:pm-arith}\ref{it:prop:pm-arith(iii)}, Lemma \ref{lem:interval-minimal-length-sets}, and Proposition \ref{prop:min-basics}\ref{it:prop:min-basics(iii)} that
$\llb 2, k \rrb \in \mathscr{L}^\m(\PPc_\fun(Y)) \subseteq \mathscr{L}^\m(\PPc_\fun(H))$.
\end{proof}
\section{Closing remarks}
This preliminary foray into minimal factorizations raises several questions. In particular, to what extent can finite subsets of $\NNb_{\ge 2}$ be realized as sets of minimal lengths of some power monoid?
Which families of subsets can be simultaneously realized by a single power monoid?

There are also questions beyond sets of lengths that can be addressed; we should expect to be able to formulate and study the ``minimal versions'' of other invariants commonly considered in Factorization Theory (unions of length sets, sets of distances, elasticities, catenary and tame degrees, etc.).
The results we have seen in the present paper suggest that the study of these types of invariants in power monoids is almost never trivial.

\section*{Acknowledgments}
We are grateful to Qinghai Zhong for helping with the proof of Theorem \ref{thm:BF-torsion}\ref{it:thm:BF-torsion(i)}; to Alfred Geroldinger for several productive discussions; and to an anonymous reviewer for comments that pushed us to improve this manuscript.
The first-named author also wishes to thank the OSU Department of Mathematics for the opportunity to travel for an extended period, and everyone at the  Institute for Mathematics and Scientific Computing at University of Graz for being gracious hosts.

\end{document}